\documentclass{cedram-aif}

\usepackage[utf8]{inputenc}
\usepackage[T1]{fontenc}
\usepackage[english]{babel}
\usepackage{amsthm}
\usepackage{amsfonts}         
\usepackage{amsmath}
\usepackage{amssymb}
\usepackage{hyperref}
\usepackage{tikz}
\usepackage{mathrsfs}  
\usetikzlibrary{matrix, arrows, decorations.pathmorphing}
\usepackage{hyperref}

\newcommand{\fref}[1]{\hyperref[{#1}]{\ref*{#1}}}

\newcommand{\Nb}{\mathbb{N}}

\newcommand{\Pb}{\mathbb{P}}

\newcommand{\Rb}{\mathbb{R}}

\newcommand{\Zb}{\mathbb{Z}}

\newcommand{\Cc}{\mathcal{C}}

\newcommand{\Ec}{\mathcal{E}}

\newcommand{\Oc}{\mathcal{O}}

\newcommand{\Uc}{\mathcal{U}}

\newcommand{\Ls}{\mathscr{L}}
\newcommand{\Ms}{\mathscr{M}}

\newcommand{\op}{\mathrm{op}}
\newcommand{\ch}{\mathrm{CH}}

\newcommand{\isomto}{\stackrel{\sim}{\rightarrow}}

\newcommand{\rank}{\mathrm{rank}}

\newcommand{\Hom}{\mathrm{Hom}}
\newcommand{\Spec}{\mathrm{Spec}}

\newcommand{\wtil}{\widetilde}

\newcommand{\gr}{\mathrm{gr}}

\newcommand{\Laz}{\mathbb{L}}
\newcommand{\A}{\mathbb{A}}
\newcommand{\Proj}{\mathbb{P}}
\newcommand{\R}{\mathbb{R}}

\newcommand{\Z}{\mathbb{Z}}

\equalenv{rem}{rema}
\equalenv{lem}{lemm}
\equalenv{cor}{coro}
\equalenv{defn}{defi}
\equalenv{con}{conj}
\equalenv{ex}{exem}

\title{Oriented Borel--Moore homologies of toric varieties}
\alttitle{Homologies orientées de Borel--Moore des variétés toriques}
\author{\firstname{Toni} \middlename{M.} \lastname{Annala}}
\address{University of British Columbia \\ 
Department of Mathematics \\
1984 Mathematics Rd \\
Vancouver, BC V6T 1Z2 \\
Canada}
\email{tannala@math.ubc.ca}

\altkeywords{homologie orienté de Borel--Moore, variété torique, cobordisme algébrique}


\date{}
\begin{document}

\begin{abstract}
We generalize the well known Künneth formula for Chow groups to an arbitrary oriented Borel--Moore homology theory satisfying localization and descent (e.g. algebraic bordism) when taking a product with a toric variety. As a corollary we obtain a universal coefficient theorem for the operational cohomology rings. We also give a new, homological, description for the homology groups of smooth toric varieties, which allows us to compute the algebraic bordism groups of some singular toric varieties. 
\end{abstract}

\begin{altabstract}
Nous généralisons la formule de Künneth bien connue pour les groupes de Chow au cas d'une théorie homologique orientée de Borel-Moore arbitraire qui vérifient des propriétés de localisation et de descente (par exemple le bordisme algébrique) pour les produits à variétés toriques. En corollaire, nous obtenons un théorème des coefficients universels pour les anneaux de cohomologie opérationnelle. Nous donnons également une nouvelle description, de nature homologique, des groupes d'homologie des variétés toriques lisses, qui nous permet de calculer les groupes de bordisme algébrique de quelques variétés toriques singulières.
\end{altabstract}

\maketitle

\tableofcontents

\section{Introduction}
In \cite{FMSS} and \cite{FS} the authors discovered the following surprising result: the operational Chow cohomology rings $\op \ch^* (X_\Delta)$ of a \emph{complete} toric variety can be naturally identified with the $\Z$-dual of the usual Chow groups $\ch_* (X_\Delta)$ as Abelian groups. Such well behavedness was not anticipated from the operational cohomology rings, which were originally thought only as a temporary substitute until some sensible cohomology theory comes to take its place. 

The proof depends crucially on another result interesting in its own right. Namely, if we are given a toric variety $X_\Delta$, then the natural Künneth map 
\begin{equation*}
\ch_* (X_\Delta) \otimes_\Z \ch_*(Y) \to \ch_* (X_\Delta \times Y)
\end{equation*}
is an isomorphism for all varieties $Y$. This is a nontrivial result as it implies, among other things, that if $Y$ and $X_\Delta$ are smooth, then all the line bundles on $Y \times X_\Delta$ can be expressed as an exterior product of line bundles on $Y$ and $X_\Delta$, and this expression is unique up to isomorphism (neither of these facts hold in general). 

A natural question to ask is if these theorems generalize to other similar theories such as the equivariant Chow groups $\ch_*^T$ or algebraic bordism $\Omega_*$. The results would be useful especially in the case of algebraic bordism, where computations are hard, and the groups are known only in a handful of cases. Structural results on the behavior of these groups are therefore important to improve our understanding.

This paper grew out of the attempts to generalize the two results to other theories than $\ch_*$. It turns out that the Künneth formula is the harder one of these, and that the original proof for the universal coefficient theorem goes through for an arbitrary theory as soon as the Künneth isomorphism property is known. Moreover, the techniques that allow us to generalize the Künneth property also offer a nice homological description of the algebraic bordism $\Omega_*(X_\Delta)$ of a smooth toric variety. The previous descriptions for the groups were \emph{cohomological} --- they were actually descriptions of the cohomology rings, which are isomorphic to the homology groups by Poincaré duality. 

The problem with the cohomological descriptions is that it is not natural to work with some of the functorial aspects of the homology groups. Namely, describing pullback operations is easy, but describing the proper pushforward maps is hard. In order to use the descent exact sequence to calculate the homology groups of singular varieties, we need to be able to effectively work with pushforwards, and hence the need for a new description. In the case of algebraic bordism, no  calculations of homology groups of singular toric varieties have been carried out previously. In fact, as far as the author is aware, the examples computed in this paper are the first examples of algebraic bordism group of any singular variety.

\subsection{Outline of the article}

We begin with a section summarizing the necessary background material on oriented Borel--Moore homology theories, and on linear varieties. Most of the proofs are skipped, and instead references are given to appropriate texts. In the first real section of this paper, Section \fref{OBMofToric}, we characterize the equivariant homology groups $B^T_*(X_\Delta)$ of a smooth toric variety, where $B_*$ is an oriented Borel--Moore homology theory. The methods are inspired by those in \cite{Kri} and \cite{KU}. The main technical work of this section is done in the first three subsections. The equivariant groups $B_*^T$ are shown to determine $B_*$, and hence we also obtain a description for $B_*(X_\Delta)$ in the case $X_\Delta$ is nonsingular. 

The description of algebraic bordism of a smooth toric variety we obtain is nice enough to allow us to effectively carry out the computation of the algebraic bordism group $\Omega_*(X_\Delta)$ for some singular toric varieties in the Section \fref{Computations} using the descent exact sequence of \cite{Karu2}. In Section \fref{Generalities}, we compare our homological description of the equivariant groups $B^T_*(X_\Delta)$ with the previously studied cohomological \emph{Stanley--Reisner} presentations of the same groups. These two sections are independent from the other results of this paper, and can be skipped.

Section \fref{StructureResults} is devoted to proving the Künneth isomorphism property for all oriented Borel--Moore homologies $B_*$ when taking a product with a toric variety, and the universal coefficient theorem that follows from it. The first of these is obtained in \fref{Kunneth}
\begin{thm*}
Let $X_\Delta$ be a toric variety, and $B_*$ an oriented Borel--Moore homology theory satisfying localization and descent. Then the Künneth map
\begin{equation*}
B_*(X_\Delta) \otimes_{B^*} B_*(Y) \to B_*(X_\Delta \times Y)
\end{equation*}
is an isomorphism for all varieties $Y$.
\end{thm*}
\noindent Interesting examples where this result holds are the equivariant Chow groups $\ch_*^G$, algebraic bordism $\Omega_*$ and the equivariant algebraic bordism $\Omega_*^G$, although we prove the result in far greater generality (see Theorems \fref{Kunneth} and \fref{EquivariantKunneth}). Note that here we are not taking the tensor product over the integers $\Z$ but over the ring $B^* := B_{-*}(pt)$ which is the natural coefficient ring of the theory $B_*$.

In Section \fref{UnivCoeffSec} we prove the universal coefficient theorem for the operational cohomology rings $\op B^*(X_\Delta)$ when $X_\Delta$ is a complete toric variety. Here we need to assume that $B_*$ has slightly stronger properties than those of an oriented Borel--Moore homology theory, namely, $B_*$ must come with refined l.c.i. pullbacks. This is needed for the construction of the operational cohomology rings.
\begin{thm*}
Let $X_\Delta$ be a complete toric variety, and $B_*$ an refined oriented Borel--Moore homology theory satisfying localization and descent. Then there is a canonical identification
\begin{equation*}
\op B^*(X_\Delta) \cong \Hom_{B^*} (B_*(X_\Delta), B^*).
\end{equation*} 
\end{thm*}
\noindent Notice how, much like in the Künneth formula, we are taking the $\Hom$ over the coefficient ring of the theory. Again, interesting examples where this result holds are $\ch_*^G$, $\Omega_*$ and $\Omega_*^G$.

\subsection{Conventions and notations}

All schemes will be separated and finite type over a field $k$ of characteristic zero. By \cite{Karu3} we can get rid of all the projectivity assumptions of algebraic bordism, and we will use this throughout the article (e.g., we have \emph{proper} pushforwards instead of projective ones and so on).

We will denote by $B_*$ a general oriented Borel--Moore homology theory on the category of $G$-schemes for a linear algebraic group $G$. Usually $G$ is either trivial or a split torus $T$. When restricted to smooth varieties $X$, $B_*$ gives an oriented cohomology theory $B^*$ in a natural way, where the contravariant functoriality is provided by the l.c.i. pullbacks, and the ring structure is provided by the intersection product. Note that cohomology has different grading: $B^*(X) = B_{n-*}(X)$ where $n$ is the dimension of $X$. 

By abuse of notation, we will denote by $B^*$ the $B_*$-homology group $B_{-*}(pt)$, which acts on all the groups $B_*(X)$ by the exterior product. It is the natural coefficient ring for the theory $B_*$. It does not matter whether we take the homological or cohomological grading for $B^*$ as the grading needs never to be explicitly mentioned.

If $G$ is a linear algebraic group, then we will denote by $B_*^G$ the $G$-equivariant version of $B_*$, whose construction is based on Totaro's approximation scheme for the classifying space $BG$. The construction of equivariant groups was carried out for a very general class of theories in \cite{Karu} which includes all oriented Borel--Moore homology theories. The groups $B^G_*$ have formally very similar properties to those of $B_*$, and indeed we can usually treat both cases simultaneously. 

\subsection{Acknowledgements}

The author would like to thank their advisor Kalle Karu for a patient introduction to algebraic cobordism, and for many helpful discussions where he pointed out multiple potential weaknesses and mistakes. The author is also grateful for the anonymous referee for multiple comments that greatly improved the exposition.

\section{Background}

In this section we are going to summarize technical background necessary for the results of this paper. This consists mostly of definitions, and the proofs are mostly omitted.

\subsection{Oriented Borel--Moore homology theories}

Let us denote by $\mathcal{C}$ a category, which is either the category of finite type separated $k$-schemes or the category of finite type separated $k$-schemes with an action by a linear algebraic group $G$ with $G$-equivariant morphisms. Both these categories have the notions of proper and l.c.i. morphisms, as well as the notion of transversality of a Cartesian square. Let us also recall the equivariant notion of a vector bundle. From now on, whenever we say \emph{scheme}, we mean a separated scheme of finite type over $k$.

\begin{defn}
Let $G$ be a linear algebraic group and $X$ a $G$-scheme, and let us denote the group action by $\mu: G \times X \to X$. A \emph{$G$-equivariant vector bundle} on $X$ consists of a locally free sheaf $\Ec$ (of constant and finite rank) on $X$ and an isomorphism
$$\theta: \mu^*(\Ec) \isomto \mathrm{pr}_2^*(\Ec)$$
of coherent sheaves on $G \times X$ satisfying a certain cocycle condition. The reader can consult \cite{MFK} Definition 1.6 for the precise definition in the case of line bundles, which immediately generalizes to vector bundles of higher rank.

As noted in \cite{MFK} Chapter 1 Section 3, there is a more geometric way of defining equivariant vector bundles. Namely, a $G$-equivariant vector bundle is
\begin{enumerate}
\item a vector bundle $\pi: E \to X$;
\item a $G$-action $\mu_E$ on the total space $E$;
\end{enumerate}
such that
\begin{enumerate}
\item[a)] $\pi: E \to X$ is $G$-equivariant;
\item[b)] the square
\begin{center}
\begin{tikzpicture}[scale=2]
\node (A1) at (0,1) {$\mathrm{pr}_2^* E = G \times E$};
\node (B1) at (1.5,1) {$E$};
\node (A2) at (0,0) {$G \times X$};
\node (B2) at (1.5,0) {$X$};

\path[every node/.style={font=\sffamily\small}]
(A1) edge[->] node[right]{$\pi'$} (A2)
(B1) edge[->] node[right]{$\pi$} (B2)
(A1) edge[->] node[above]{$\mu_E$} (B1)
(A2) edge[->] node[above]{$\mu$} (B2)
;
\end{tikzpicture}
\end{center}
is a morphism of vector bundles (on different base schemes).
\end{enumerate}

Given a $G$-equivariant vector bundle $E$ on $X$, we may form the corresponding \emph{equivariant projective bundle} $\Pb(E) \to X$. The \emph{rank} of $\Pb(E)$ is by definition $\rank(E)-1$. 
\end{defn}

\begin{defn}\label{OBMHomologyDef}
Let everything be as above. Denote by $\mathcal{C}'$ the subcategory of $\mathcal{C}$ whose objects are the objects of $\mathcal{C}$, and whose morphisms are the proper morphisms in $\mathcal{C}$. An \emph{oriented Borel--Moore homology theory} consists of


\begin{enumerate}
\item[($D_1$)] a covariant functor $X \mapsto B_*(X)$ from $\Cc'$ to the category of graded Abelian groups;


\item[($D_2$)] (\emph{l.c.i. pullbacks}) a group homomorphism
$$f^*: B_*(Y) \to B_{*+d}(X)$$
for all l.c.i. morphisms $f: X \to Y$ of relative dimension $d$;

\item[($D_3$)] (\emph{exterior product}) a graded bilinear pairing
\begin{equation*}
\times: B_*(X) \times B_*(Y) \to B_*(X \times Y)
\end{equation*}
for all $X$ and Y, which is commutative, associative, and has a unit $1 \in B_0(\Spec(k))$.
\end{enumerate} 
Notice that this data allows us to define the \emph{first Chern class operator} 
$$c_1(\Ls) := s^* \circ s_*: B_*(X) \to B_{*-1}(X)$$
of a line bundle $\mathscr{L}$ on $X$, where $s$ is the zero section $X \to \Ls$. Of course, if $\Cc$ is the category of equivariant schemes, then a line bundle implicitly means an equivariant line bundle.  

The data of an oriented Borel--Moore homology theory are required to satisfy the following axioms:
\begin{enumerate}
\item[($Ad$)] (\emph{additivity}): the natural map 
\begin{equation*}
\bigoplus_{i=1}^n B_*(X_i) \to B_*\left( \coprod_{i=1}^n X_i \right)
\end{equation*}
is an isomorphism;

\item[($BM_1$)] the l.c.i. pullbacks are contravariantly functorial;

\item[($BM_2$)] given a transverse Cartesian square
\begin{center}
\begin{tikzpicture}[scale=2]
\node (A1) at (0,1) {$X'$};
\node (B1) at (1,1) {$Y'$};
\node (A2) at (0,0) {$X$};
\node (B2) at (1,0) {$Y$};

\path[every node/.style={font=\sffamily\small}]
(A1) edge[->] node[right]{$f'$} (A2)
(B1) edge[->] node[right]{$f$} (B2)
(A1) edge[->] node[above]{$g'$} (B1)
(A2) edge[->] node[above]{$g$} (B2)
;
\end{tikzpicture}
\end{center}
with $f$ l.c.i. and $g$ proper, then 
$$f^* \circ g_* = g'_* \circ f'^*$$
(note that being l.c.i. is stable under transverse pullbacks);

\item[($BM_3$)] given proper morphisms $f: X \to Y$ and $g: X' \to Y'$, then 
\begin{equation*}
(f \times g)_* (\alpha \times \beta) = f_*(\alpha) \times g_*(\beta),
\end{equation*} 
and similarly if $f$ and $g$ are l.c.i., then 
\begin{equation*}
(f \times g)^* (\alpha \times \beta) = f^*(\alpha) \times g^*(\beta);
\end{equation*} 

\item[($PB$)] (\emph{projective bundle formula}) given a projective bundle $\pi: \Pb(E) \to X$, $r := \rank(E) - 1$, then the morphisms
\begin{equation*}
c_1(\Oc(1))^i \circ \pi^*: B_{*-r+i}(X) \to B_*(\Pb(E))
\end{equation*}
induce an isomorphism
\begin{equation*}
\bigoplus_{i=0}^r B_{*-r+i}(X) \to B_*(\Pb(E)); 
\end{equation*}

\item[($EH$)] (\emph{extended homotopy property}) given a vector bundle $V \to X$ of rank $r$ and $p: E \to X$ a $V$-torsor, then the pullback morphism 
\begin{equation*}
p^*: B_*(X) \to B_{*+r}(E)
\end{equation*}
is an isomorphism.
\end{enumerate}
\end{defn}

\begin{rem}
In the book \cite{LM} it is also required that $B_*$ should satisfy a cellular decomposition axiom. As we do not need this for anything, we will omit it from the definition.
\end{rem} 

\begin{rem}
The exterior product makes $B_*(X)$ a $B^* = B_*(pt)$-module for all $X$. It follows from the requirements imposed on the exterior product $\times$ that the pushforwards and pullbacks are $B^*$-linear maps. Moreover, from the commutativity and the associativity assumptions it follows that the exterior product associated to $X \times Y$ is $B^*$-bilinear. Hence we have a natural Künneth morphism
\begin{equation*}
B_*(X) \otimes_{B^*} B_*(Y) \to B_*(X \times Y),
\end{equation*}
which is one of the protagonists of the paper.
\end{rem}

Well known examples of such theories include algebraic bordism $\Omega_*$ (see \cite{LM}) and the Chow groups $\ch_*$ (see \cite{Ful}). One can refine this definition to require the theory to have \emph{refined l.c.i.} pullbacks. This means that for a l.c.i. morphism $f: X \to Y$ of relative dimension $d$ and \emph{any} morphism $g: Y' \to Y$, we get a pullback map $f^!_g: B_*(Y') \to B_{*+r}(X')$, where $X'$ is the pullback of $X$ along $g$. These refined pullbacks are required to satisfy certain compatibility conditions, see \cite{Ful} or \cite{Karu} for details. Both the Chow groups and algebraic bordism have refined l.c.i. pullbacks. Next we consider two useful extra properties one can require from an oriented Borel--Moore homology theory.

\begin{defn}
Let $B_*$ be an oriented Borel--Moore homology theory. We say that $B_*$ \emph{satisfies localization} if for all closed embeddings $i: Z \hookrightarrow X$, the sequence 
\begin{equation*}
B_*(Z) \stackrel{i_*}{\to} B_*(X) \stackrel{j^*}{\to} B_*(U) \to 0
\end{equation*}
is exact, where $j: U \hookrightarrow X$ is the inclusion of the open complement of $Z$.
\end{defn}

Recall that an \emph{envelope} $\wtil X \to X$ is a morphism such that any irreducible subvariety of $X$ is birationally mapped onto by an irreducible subvariety of $\widetilde X$.

\begin{defn}
Let $B_*$ be an oriented Borel--Moore homology theory. We say that $B_*$ \emph{satisfies descent} if for any proper envelope $\pi: \wtil X \to X$, the induced sequence
\begin{equation*}
B_*(\widetilde X \times_X \widetilde X) \xrightarrow{\mathrm{pr}_{1*} - \mathrm{pr}_{2*}} B_*(\widetilde X) \xrightarrow{\pi_*} B_*(X) \to 0
\end{equation*}
is exact.
\end{defn}

It is not known whether or not these properties hold for a general Borel--Moore homology theory, but they are known to hold for $\Omega_*$ by \cite{LM,Karu2}, and therefore also for all oriented Borel--Moore homology theories obtained from algebraic bordism by a change of coefficients.

\subsection{Formal group laws}

A \emph{formal group law} over a commutative ring $A$ is an element $F \in A [[x,y]]$, satisfying the following properties:
\begin{enumerate}
\item \emph{Neutral element:} $F(x,0) = x$ and $F(0,y)=y$.
\item \emph{Commutativity:} $F(x,y) = F(y,x)$.
\item \emph{Associativity:} $F(x,F(y,z)) = F(F(x,y),z)$.
\end{enumerate}
It is immediate from the first two restrictions that $F$ must be of form
\begin{equation*}
F(x,y) = x + y + \alpha_{11}xy + \alpha_{21}x^2y + \alpha_{12}xy^2  \cdots
\end{equation*}
where $\alpha_{ij} \in A$, and $\alpha_{ij} = \alpha_{ji}$. The third axiom induces more complicated relations between the coefficients.

It is not hard to show that such an $F$ has a \emph{formal inverse series}, i.e., a power series $F_-$ in one variable such that 
\begin{equation*}
F(x, F_-(x)) = F(F_-(x), x) = 0.
\end{equation*}
It is often more convenient to denote $F(x,y)$ by $x +_F y$ and $F_-(x) = -_Fx$. Moreover, the repeated addition $x +_F \cdots +_F x$ (here $x$ appears $n$ times) can be denoted by $n \cdot_F x$. It is not hard to show that these behave as one would expect, i.e., $n \cdot_F x +_F n \cdot_F y = n \cdot_F (x+_Fy)$, $-n \cdot_F x = n \cdot_F (-x)$ and so on.

There is universal such a group law. Consider the infinitely generated $\Z$-algebra $\Z[a_{11}, a_{21}, a_{12}, ...]$ with the minimal relations making the power series
\begin{equation*}
x + y + a_{11} xy + a_{21} x^2y + a_{12} xy^2 + \cdots 
\end{equation*}
a formal group law. The resulting ring is known as the \emph{Lazard ring} $\Laz$, and it is characterized by the universal property that the set of ring homomorphisms $\psi: \Laz \to A$ is in natural one-to-one correspondence with the set of formal group laws with coefficients in $A$ (the coefficients of the corresponding formal group law are given by $\psi(a_{ij})$). If we set the degree of $a_{ij}$ to be $i+j -1$, then all the relations respect grading, so we see that the Lazard ring has a natural grading.

Formal group laws describe the behavior of Chern classes of line bundles in oriented Borel--Moore homology theories $B_*$. Namely, it turns out that there always exists a formal group law $F \in B_*[[x,y]]$ so that given line bundles $\Ls$ and $\Ms$ on $X$, we have
$$c_1(\Ls \otimes \Ms) = F(c_1(\Ls), c_1(\Ms)) : B_*(X) \to B_{*-1}(X).$$ 
Usually infinite expressions are not allowed for our theories, so in order to make sense of the formal group law property one must require that the Chern classes are nilpotent in a suitable sense. However, sometimes infinite expressions make perfect sense (namely in equivariant theories), and then we can make sense of the formal group law topologically (i.e., the series will converge to something).

For Chow groups $\ch_*$, the formal group law is known to be the \emph{additive} formal group law $F(x,y) = x + y$. A more complicated example is the algebraic bordism $\Omega_*$ --- the universal oriented Borel--Moore homology theory --- whose formal group law is the universal formal group law over the Lazard ring defined above.

\subsection{Linear and $G$-linear varieties}

In the paper \cite{Tot} it was shown that the Chow groups satisfy the Künneth formula for products of arbitrary variety and a \emph{linear scheme}. Unfortunately the proof made use of higher Chow groups, and therefore it cannot be generalized to, say, algebraic bordism. However, part of the argument can be salvaged, and the purpose of this section is to record this. First, however, we need a definition.


\begin{defn}
Let $G$ be a linear algebraic group. A \emph{$G$-linear variety} is a $G$-variety obtained by the following inductive procedure.
\begin{enumerate}
\item A $G$-representation (thought as a variety) is $G$-linear.

\item If we have a $G$-equivariant closed embedding $Z \hookrightarrow X$ of $G$-linear varieties, then the complement $X - Z$ is $G$-linear.

\item If we have a $G$-equivariant filtration $\emptyset = X_{-1} \subset X_0 \subset \cdots \subset X_r = X$ such that $X_i - X_{i-1}$ is a $G$-linear variety for all $i=0..r$, then also $X$ is $G$-linear.  
\end{enumerate}
\end{defn}

Note that in the special case where $G$ is trivial, the above definition recovers the notion of a \emph{linear variety}.

\begin{prop}\label{ModifiedTotaro}
Let $X$ be a $G$-linear variety, $Y$ be an arbitrary $G$-scheme, and let $B_*$ be an oriented Borel--Moore homology theory (on $G$-schemes) satisfying localization. Then the Künneth map
\begin{equation*}
B_*(X) \otimes_{B^*} B_*(Y) \to  B_*(X \times Y)
\end{equation*}
is surjective.
\end{prop}
\begin{proof}
The fact that the map is surjective when taking a product with a $G$-representation follows from the extended homotopy property and compatibility properties of exterior products and pullbacks. 

Assume that we have a closed inclusion $Z \hookrightarrow X$ such that the claim holds for $X - Z$ and $Z$. Then the localization exact sequence yields us a diagram
\begin{center}
\begin{tikzpicture}[scale=2]
\node (A1) at (0,1) {$B_*(Z) \otimes_{B^*} B_*(Y)$};
\node (B1) at (2.2,1) {$B_*(X) \otimes_{B^*} B_*(Y)$};
\node (C1) at (4.4,1) {$B_*(X-Z) \otimes_{B^*} B_*(Y)$};
\node (D1) at (5.8,1) {$0$};
\node (A2) at (0,0) {$B_*(Z \times Y)$};
\node (B2) at (2.2,0) {$B_*(X \times Y)$};
\node (C2) at (4.4,0) {$B_*((X - Z) \times Y)$};
\node (D2) at (5.8,0) {$0$};

\path[every node/.style={font=\sffamily\small}]
(A1) edge[->] (A2)
(B1) edge[->] (B2)
(C1) edge[->] (C2)
(A1) edge[->] (B1)
(B1) edge[->] (C1)
(C1) edge[->] (D1)
(A2) edge[->] (B2)
(B2) edge[->] (C2)
(C2) edge[->] (D2)
;
\end{tikzpicture}
\end{center}
where the leftmost and rightmost vertical maps are surjections. From the 4-lemma it follows that also the middle vertical map is surjective. This shows that surjectivity is preserved in the third operation defining linear varieties. Showing that it is preserved in the second operation is similar but easier, so the claim follows.
\end{proof}

\section{Homology groups of toric varieties}\label{OBMofToric}

The main purpose of this section is to study the structure of $B_*$-homology groups of toric varieties. Throughout the section, unless otherwise specified, $T$ will be a fixed split torus of rank $n$ and $B_*$ will be an arbitrary oriented Borel--Moore homology theory on the category $\Cc$ of $k$-schemes (without a group action) \emph{satisfying localization}. Our strategy is to first study the equivariant homology groups $B_*^T$, whose construction is recalled in the first subsection, and then reduce the study of non-equivariant groups $B_*$ to the equivariant ones.  The results of this section are crucial for the Künneth formula and duality results proven in Section \fref{StructureResults}.

Let $N \cong \Z^n$ be the lattice of one-parameter subgroups of $T$, and let us denote by $N_\Rb$ the vector space defined to be the tensor product $\Rb \otimes N$. Recall that a \emph{toric variety} $X_\Delta$ for $T$ is determined by a fan $\Delta$ consisting of rational strictly convex polyhedral cones in $N_\Rb$ (see \cite{Ful2} for more details). The dual lattice $\Hom_\Z(N, \Z)$, which is the \emph{character lattice} of $T$, is denoted by $M$, and we define $M_\Rb := \Rb \otimes M$. Let us also fix a $\Zb$-basis $e_1,...,e_n$ of $N$, and let us denote the corresponding dual basis of $M$ by $e^*_1,...,e^*_n$.

\subsection{Review of the construction and basic properties of equivariant groups}

The construction of the equivariant version $B^G_*$ of $B_*$ is based on the approximation scheme of Totaro. This has led to successful study of equivariant Chow groups \cite{EG} and algebraic bordism \cite{Des}\cite{Kri}\cite{KU}, among other theories. Recently in \cite{Karu}, the construction was carried out in a very general setting of so called ROBM pre-homology theories satisfying certain conditions. The same construction works for oriented Borel--Moore homology theories, and it is the construction we are going to use here. We note that in order to make sense of the formal group laws in the theory $B_*^G$, one is forced to consider it as a topological group, whose topology is given by the filtration naturally arising from the definition. This is because equivariant version of the theory no longer has to satisfy dimension axiom, and hence Chern classes of line bundles may fail to be nilpotent.

The following definition will be of crucial importance:

\begin{defn}
Let $G$ be a linear algebraic group. A \emph{good system of representations} for $G$ consists of pairs $(V_i,U_i)$, where $V_i$ is a $G$-representation and $U_i \subset V_i$ is an open subscheme. These are required to satisfy the following conditions
\begin{enumerate}
\item $G$ acts freely on $U_i$, and the geometric quotient $U_i / G$ exists as a quasi-projective $k$-scheme;
 
\item the $G$-representation $V_{i+1}$ splits as $V_i \oplus W_i$ for all $i$;

\item $U_i \oplus W_i \subset U_{i+1}$; 

\item the codimension of $V_i - U_i$ in $V_i$ is strictly smaller than that of $V_j - U_j$ in $V_j$ whenever $i<j$. 
\end{enumerate}
\end{defn}

\begin{ex}\label{StdSystemOfRepsForT}
An example of a good system of representations for $T$ would be 
$$\big((\A^{i+1})^n, (\A^{i+1}-0)^n\big),$$ 
where the $j^{th}$ coordinate of the torus acts diagonally on the $j^{th}$ copy of $\A^{i+1}$. The quotient can be identified as
$$(\A^{i+1}-0)^n / T \cong (\Proj^i)^n,$$
which is projective.
\end{ex}

Before defining the equivariant groups let us recall that the \emph{mixed quotient} $X \times^G Y$ of two $G$-schemes is by definition (if it exists) the geometric quotient 
$$(X \times Y) / G$$
where $G$ acts on $X \times Y$ by the \emph{anti-diagonal} action (i.e., we invert the action on $X$).

\begin{defn}[cf. \cite{HML} Definition 12, \cite{Karu} Section 4.2]
Let $G$ be a linear algebraic group. The equivariant groups $B^G_*(X)$ of a $G$-scheme $X$ are defined as 
$$B^G_k(X) := \varprojlim_i B_{k + \dim(U_i) - g}(U_i \times^G X),$$
where $U_i$ are in any good system of representations of $G$, and $g = \dim(G)$. Note that the connecting morphisms of the inverse system are given by the pullbacks along the induced l.c.i. morphisms $U_i \times^G X \hookrightarrow U_j \times^G X$. We also set
$$B^G_*(X) := \bigoplus_{i \in \Zb} B^G_i(X).$$
For a proof that these groups are well defined, see \cite{HML} Theorem 16 and \cite{Karu} Proposition 4.4.
\end{defn}

We will mostly deal with torus equivariant theories. However, suppose $B_*$ is an oriented Borel--Moore homology theory on $G$-schemes for some linear algebraic group $G$. If we would like to define $T$-equivariant versions of $B_*$, then we would have to verify that the arguments of \cite{HML} proving that $B^T_*$ is well defined extend to the $G$-equivariant setting. It is much easier to use the following result, which proves the well definedness in the most important special case.

\begin{prop}\label{GTEquivariant}
Let $G$ be a linear algebraic group, and let $X$ be a $T \times G$-scheme. Then there is a natural isomorphism
$$B^{T \times G}_*(X) \cong (B^G)^T_*(X).$$
In particular, the $T$-equivariant version of $B^G_*$ is well defined.
\end{prop}
\begin{proof}
Let $(V_i, U_i)$ be a good system of representations for $T$, and $(V'_i, U'_i)$ a good system of representations for $G$, so that $(V_i \times V'_i, U_i \times U'_i)$ is a good system of representations for $T \times G$. Then, denoting by $g$ the dimension of $G$, we have
\begin{align*}
B^{T \times G}_k(X) &= \varprojlim_i B_{k + \dim(U_i) + \dim(U'_i) - n - g}\big((U_i \times U'_i) \times^{T \times G} X \big) \\
&= \varprojlim_i B_{k + \dim(U_i) + \dim(U'_i) - n - g}\big(U_i \times^T (U'_i \times^G X) \big).
\end{align*}
By cofinality, this is equivalent to
$$\varprojlim_i \varprojlim_j B_{k + \dim(U_i) + \dim(U'_j) - n - g}\big(U_i \times^T (U'_j \times^G X) \big),$$
which is by definition just
$$\varprojlim_i  B^G_{k + \dim(U_i) - n} (U_i \times^T X),$$
finishing the proof.
\end{proof}

We record the fact that computing the $T$-equivariant group of the point is trivial.

\begin{prop}
There is a natural isomorphism
$$B^T_*(pt) \cong B_*(pt)[[\xi_1, ..., \xi_n]]_\gr,$$
where $\xi_i$ is the first $T$-equivariant Chern class of the character line bundle corresponding to the basis element $e^*_i$ for $M$. The right hand side is the \emph{graded power series ring} (each $\xi_i$ having degree $-1$), consisting of finite linear combinations of power series of constant degree, i.e., of form
$$\sum_{\overline m \in \Nb^n} a_{\overline m} \cdot \xi_1^{m_1} \cdots \xi_n^{m_n}$$
with each $a_{\overline m} \cdot \xi_1^{m_1} \cdots \xi_n^{m_n}$ homogeneous and of constant degree.
\end{prop}
\begin{proof}
Indeed, one uses the good system of Example \fref{StdSystemOfRepsForT} to compute this group, and the rest follows from the projective bundle formula and the definition of the first (equivariant) Chern class in an oriented Borel--Moore homology theory.
\end{proof}

We immediately obtain the following result. Notice how it highlights that if $B_*$ has a complicated formal group law, changing the basis of $N$ (and hence the dual basis of $M$) can lead to complicated formulas.

\begin{prop}\label{ChernClassesOfCharactersProp}
Consider a character 
$$m = a_1 e^*_1 + \cdots + a_n e^*_n \in M,$$
and let $V_m$ be the corresponding $T$-equivariant line bundle on $pt$. Then
$$c_1(V_m) = a_1 \cdot_F \xi_1 +_F \cdots +_F a_n \cdot_F \xi_n \in B_{-1}(pt),$$
where $F$ is the formal group law of the theory $B_*$.
\end{prop}

\subsubsection*{The oriented Borel--Moore homology structure of the equivariant groups}

Here we recall some useful formal properties of the equivariant groups $B^G_*$. The justifications for all the claims can be essentially found in Section 4 of \cite{HML} (see also Section 5 of \cite{HML} and Section 4.3 of \cite{Karu}). 

We claim that the equivariant homology groups form an oriented Borel--Moore homology theory on the category of $G$-schemes in the sense of Definition \fref{OBMHomologyDef}. Indeed, most of the claims are pretty straightforward:
\begin{itemize}
\item $B^G_*$ has functorial proper pushforwards and l.c.i. pullbacks along equivariant maps constructed as a limit in the obvious way;

\item the axioms ($Ad$), ($BM_1$) and ($BM_2$) follow from the obvious limiting argument; 

\item the axioms ($PB$) and ($EH$) require understanding also a little bit of the geometry of the situation, but they also follow (see \cite{HML} Sections 4.3 and 4.4). 
\end{itemize}
Moreover, since the transition maps of any inverse system used to construct $B^G_*$ are surjective, it satisfies the Mittag--Leffler condition, and therefore
\begin{itemize}
\item $B^G_*$ satisfies localization, and if $B_*$ satisfies descent, then so does $B^G_*$.
\end{itemize}

The construction of the exterior product is slightly more subtle. Let $(V_i, U_i)$ be a good system of $G$-representations, and let $X$ and $Y$ be $G$-schemes. Then the obvious morphism
$$\phi: (U_i \times U_i) \times^G (X \times Y) \to (U_i \times^G X) \times (U_i \times^G Y)$$
is smooth, and we define the exterior product as the limit of compositions
\begin{center}
\begin{tikzpicture}[scale=1]
\node (A1) at (0,1) {$B_*(U_i \times^G X) \times B_*(U_i \times^G Y)$};
\node (B1) at (7,1) {$B_*((U_i \times^G X) \times (U_i \times^G Y))$};
\node (A2) at (7,-0.5) {$ B_*((U_i \times U_i) \times^G (X \times Y));$};

\path[every node/.style={font=\sffamily\small}]
(B1) edge[->] node[right]{$\phi^*$} (A2)
(A1) edge[->] node[above]{$\times$} (B1)
;
\end{tikzpicture}
\end{center}
we note that also $(V_i \times V_i, U_i \times U_i)$ is a good system of $G$-representations. With these definitions
\begin{itemize}
\item the assumptions of ($D_3$) on $\times$ and the axiom ($BM_3$) are satisfied for $B^G_*$. 
\end{itemize}
This concludes our analysis.



\subsubsection*{First properties of the torus equivariant groups}

We begin listing properties of $B_*^T$ that are more or less direct from the definition.

\begin{prop}\label{TrivialAction}
Let $X$ be a $k$-scheme with a trivial $T$-action. Then
\begin{equation*}
B_*^T(X) =  B^*_T \widehat \otimes_{B^*} B_*(X).
\end{equation*}
\end{prop}
\begin{proof}
As the action of $T$ on $X$ is trivial, we can identify $X \times^T U_i$ with $X \times (\Proj^i)^n$. Using the projective bundle formula, we see that
\begin{equation*}
B^* [\xi_1,...,\xi_n] / (\xi_1^{i+1},...,\xi_n^{i+1}) \otimes_{B^*} B_*(X) \to B_{* + ni}(U_i \times^T X)
\end{equation*}
is an isomorphism. The limit of the left groups is going to be linear combinations of power series of form 
\begin{equation*}
\sum_{i_1,...,i_n} \xi_1^{i_1} \cdots \xi_n^{i_n} b_{i_1,...,i_n}
\end{equation*}
where $i_1,...,i_n$ run over all the natural numbers, and 
\begin{equation*}
b_{i_1,...,i_n} \in B_{k + i_1 + \cdots + i_n} (X)
\end{equation*} 
for fixed $k$. This coincides with the completed graded tensor product of the two, where $B_*(X)$ is taken to have the trivial filtration.
\end{proof}

Another result easily proven, which is a special case of more general Morita isomorphism principle, is the following.

\begin{prop}
Let $X$ be a $T$-scheme. Then $B_{*+n}^T(T \times X)$, where the product variety has the diagonal action, is naturally isomorphic to $B_*(X)$.
\end{prop}
\begin{proof}
The map $T \times X \to T \times X_t$ defined by $(t,x) \mapsto (t,t^{-1}x)$ identifies $T \times X$, with the original action on $X$, with $T \times X_t$, where $X_t$ is $X$ with the trivial $T$-action. Therefore for all $U_i$ we have isomorphisms
\begin{equation*}
U_i \times^T (T \times X) \cong U_i \times^T (T \times X_t)  = (U_i \times^T T) \times X_t = U_i \times X_t.  
\end{equation*}
As the $U_i$ can be chosen to be $(\A^{i+1} - 0 )^n$ (see Example \fref{StdSystemOfRepsForT}), we are done if we can show that 
\begin{equation*}
B_{*+i+1}((\A^{i+1} - 0 ) \times Y) \cong B_{*}(Y)
\end{equation*}
for any $k$-scheme $Y$. From the extended homotopy property we can deduce that the smooth pullback map $B_*(Y) \to B_{*+i+1}(\A^{i+1} \times Y)$ is an isomorphism. On the other hand, the first map in the localization sequence
\begin{equation*}
B_*(Y) \stackrel{s_*} \to B_*(\A^{i+1} \times Y) \to B_* ((\A^{i+1} - 0) \times Y) \to 0
\end{equation*}
is zero because $\A^{i+1} \times Y$ is a trivial vector bundle over $Y$, $s^*s_*$ corresponds to its top Chern class (which vanishes), and because $s^*$ is an isomorphism. This shows that the pullback map $B_*(Y) \to B_{*+i+1}((\A^{i+1}-0) \times Y)$ is an isomorphism for all $Y$, and as a consequence the pullback map induces an isomorphism
\begin{equation*}
B_*(X) \to B_{*+n(i+1)}(U_i \times X). \qedhere
\end{equation*}
\end{proof}


The following result is the first case of the Künneth formula (Theorem \fref{Kunneth}). Notice how it also gives a nice connection between certain quotient groups of $B^T_*$ and $T$-varieties obtained from $X$.

\begin{prop}\label{EquivariantChernOfRepresentations}
Let $W$ be a $T$-representation of rank $r$. Then $B^T_{* + r}((W-0) \times X)$ is the quotient of $B^T_*(X)$ by the image of the top equivariant Chern class of $W$. In particular, the Künneth morphism 
\begin{equation*}
B^T_*(W-0) \otimes_{B^*_T} B^T_*(X) \to B^T_*((W-0) \times X)
\end{equation*}
is an isomorphism.
\end{prop}
\begin{proof}
As the maps $U_i \times^T (W \times X) \to U_i \times^T X$ are vector bundles, the Künneth isomorphism 
\begin{equation*}
B^T_*(W) \otimes_{B^*_T} B^T_*(X) \to B^T_*(W \times X)
\end{equation*}
follows from the extended homotopy property. Indeed, it is easy to show using only the basic properties of Borel--Moore exterior product that taking the exterior product with the fundamental class of a vector bundle exactly coincides with the associated pullback morphism. 

The localization sequence yields the following commutative diagram:
\begin{center}
\begin{tikzpicture}[scale=2]
\node (A1) at (0,1) {$B_*^T(pt) \otimes_{B^*_T} B_*^T(X)$};
\node (B1) at (2.2,1) {$B_*^T(W) \otimes_{B^*_T} B_*^T(X)$};
\node (C1) at (4.4,1) {$B_*^T(W-0) \otimes_{B^*_T} B_*^T(X)$};
\node (D1) at (5.8,1) {$0$};
\node (A2) at (0,0) {$B_*^T(X)$};
\node (B2) at (2.2,0) {$B_*^T(W \times X)$};
\node (C2) at (4.4,0) {$B_*^T((W - 0) \times X)$};
\node (D2) at (5.8,0) {$0$};

\path[every node/.style={font=\sffamily\small}]
(A1) edge[->] node[right]{$\cong$} (A2)
(B1) edge[->] node[right]{$\cong$} (B2)
(C1) edge[->] (C2)
(A1) edge[->] (B1)
(B1) edge[->] (C1)
(C1) edge[->] (D1)
(A2) edge[->] (B2)
(B2) edge[->] (C2)
(C2) edge[->] (D2)
;
\end{tikzpicture}
\end{center}
which gives the Künneth-formula for $(W-0) \times X$ by $5$-lemma. To prove the last remaining claim, it suffices to consider the localization sequence 
\begin{equation*}
B_*^T(pt) \stackrel{i_*}{\to} B_*^T(W) \to B_*^T(W - 0) \to 0.
\end{equation*}
As the zero-section pullback $i^*$ is an isomorphism, and as $i^*i_*$ corresponds to the top equivariant Chern class of the bundle $W$, we can identify $B_*^T(W - 0)$ with the quotient of $B_*^T(pt)$ by the image of $c_r(W)$ together with a degree shift.
\end{proof}

As an immediate corollary, we obtain a generalization of an analogue of a statement in \cite{Kri} for the algebraic bordism $\Omega_*$:
\begin{cor}\label{KunnethForT}
The natural surjection $B^*_T \to B^*$, obtained by setting $\xi_1, ..., \xi_n$ to be zero, gives an isomorphism
\begin{equation*}
B^* \otimes_{B^*_T} B_*^T(X) \isomto B_*(X).
\end{equation*}
\end{cor}
\begin{proof}
We already know that $B^T_{* + n} (T \times X) = B_*(X)$, where $n$ is the rank of the torus $T$. Moreover, $T \cong (V_1 - 0) \times \cdots \times (V_n - 0)$ where $V_i$ are the standard one dimensional coordinate representations of $T$, and hence by the previous lemma taking the product with $T$ corresponds algebraically to setting the variables $\xi_i$ to be zero. But setting $\xi_1,...,\xi_n=0$ is exactly how one obtains the natural map $B^*_T \to B^*$, so we are done.
\end{proof}

Therefore the equivariant groups $B^T_*$ determine the ordinary groups $B_*$. If the $T$-action is trivial, this does not help much, but as we shall see soon, a natural action can help very much in determining the structure.

\subsection{A decomposition theorem for smooth toric varieties}

We now turn our attention to toric varieties. Throughout this section, $X_\Delta$ will denote a toric variety for $T$. The following lemma will provide a basis for the decomposition theorem:

\begin{lem}\label{EquivInj}
Suppose $X_\Delta$ is nonsingular, and let $\sigma \in \Delta$ be a maximal cone (so that the orbit $O_\sigma$ is closed). Then the inclusion $i: O_\sigma \to X_\Delta$ induces an injection 
\begin{equation*}
i_*: B_*^T(O_\sigma) \to B^T_*(X_\Delta).
\end{equation*}
\end{lem}
\begin{proof}
Without loss of generality we may assume that the cone $\sigma$ is generated by $e_1,...,e_r$ so that the open set $\mathcal{U}_\sigma$ corresponding to $\sigma$ is 
\begin{equation*}
\Spec (k[x_1,...,x_r,x_{r+1}^{\pm 1}, ..., x_n^{\pm 1}]) = \A^{r} \times O_\sigma.
\end{equation*}
Denote by $j$ the inclusion $O_\sigma \to \mathcal{U}_\sigma$, and note that this can be identified with the zero section $O_\sigma \to W \times O_\sigma$ where $W$ is the $T$-representation with action
\begin{equation*}
(\lambda_1,...,\lambda_n).(w_1,...,w_r) = (\lambda_1 w_1,...,\lambda_r w_r).
\end{equation*}
We know that $j^* j_*$ corresponds to the equivariant top Chern class of $W$. As $W$ splits into the direct sum of the natural coordinate representations $V_1 \oplus \cdots \oplus V_r$, the top Chern class of $W$ is just $c_1(V_1) \cdots c_1(V_r)$, i.e., multiplication by $\xi_1 \cdots \xi_r$. On the other hand, by Proposition \fref{EquivariantChernOfRepresentations} $B^T_* (O_\sigma)$ is isomorphic to $B^T_* / (\xi_{r+1},...,\xi_n)$ with a shift, from which we conclude that the action of the top Chern class of $W$ is an injective morphism $B^T_* (O_\sigma) \to B^T_* (O_\sigma)$, and hence the map $j_*$ must be injective as well.

We can use this to show that also $i_*$ is injective. Denote by $u$ the natural open inclusion $\mathcal U_\sigma \to X_\Delta$. Now the transverse Cartesian square 
\begin{center}
\begin{tikzpicture}[scale=2]
\node (A1) at (0,1) {$O_\sigma$};
\node (B1) at (1,1) {$\mathcal{U}_\sigma$};
\node (A2) at (0,0) {$O_\sigma$};
\node (B2) at (1,0) {$X_\Delta$};

\path[every node/.style={font=\sffamily\small}]
(A1) edge[->] node[right]{$1$} (A2)
(B1) edge[->] node[right]{$u$} (B2)
(A1) edge[->] node[above]{$j$} (B1)
(A2) edge[->] node[above]{$i$} (B2)
;
\end{tikzpicture}
\end{center}
tells us that $j_* = u^*i_*$, and hence the injectivity of $i_*$ follows from that of $j_*$.
\end{proof}

We record an immediate corollary.

\begin{cor}\label{EmbeddingsToToricAreInjective}
Let $i: Z \to X_\Delta$ be a closed equivariant embedding to a smooth toric variety $X_\Delta$, i.e., $Z$ is a closed subvariety that is a union of orbits. Then the induced proper pushforward map $i_*: B_*^T(Z) \to B_*^T(X_\Delta)$ is injective.
\end{cor}
\begin{proof}
We proceed by induction on the number of orbits in $Z$, case 0 being trivial. Let $O$ be a minimal orbit in $Z$. As $Z$ is closed inside $X_\Delta$, we see that $O$ is a minimal orbit inside $X_\Delta$ as well; denote by $U$ and $V$ the open complements of $O$ in $X_\Delta$ and $Z$ respectively. By the previous lemma, the diagram
\begin{center}
\begin{tikzpicture}[scale=2]
\node (A1) at (0,1) {$B_*^T(O)$};
\node (B1) at (2,1) {$B_*^T(Z)$};
\node (C1) at (4,1) {$B_*^T(V)$};
\node (D1) at (5,1) {$0$};
\node (A2) at (0,0) {$B_*^T(O)$};
\node (B2) at (2,0) {$B_*^T(X_\Delta)$};
\node (C2) at (4,0) {$B_*^T(U)$};
\node (D2) at (5,0) {$0$};
\node (00) at (-1,0) {$0$};

\path[every node/.style={font=\sffamily\small}]
(A1) edge[->] node[right]{$1$} (A2)
(B1) edge[->] (B2)
(C1) edge[->] (C2)
(A1) edge[->] (B1)
(B1) edge[->] (C1)
(C1) edge[->] (D1)
(A2) edge[->] (B2)
(B2) edge[->] (C2)
(C2) edge[->] (D2)
(00) edge[->] (A2)
;
\end{tikzpicture}
\end{center}
has exact rows. By induction the rightmost vertical map is an injection. It then follows from diagram chasing that the middle vertical map is an injection as well.
\end{proof}

We can use the Lemma \fref{EquivInj} to arrive at the following decomposition result. Suppose we have a nonsingular toric variety $X_\Delta$. We can remove all the cones from $\Delta$ one by one by choosing an arbitrary maximal cone and removing its interior from the fan. By the previous result, at each step we have a short exact sequence
\begin{equation*}
0 \to B_*^T(O_\sigma) \to B_*^T(X_\Delta) \to B_*^T(X_{\Delta'}) \to 0 
\end{equation*}
where $B_*^T(O_\sigma)$ is isomorphic to shifted copy of $B^*_T / (\xi'_1,...,\xi'_r)$, where $\xi'_j$ are the first Chern classes of character line bundles corresponding to a $\Zb$-basis of characters in $M$ orthogonal to $\sigma$. (For more details, see the Section \fref{LinearLines} following this section).

This is very much in line with the structural results obtained in \cite{FS} for Chow groups of toric varieties. Denoting by $V_\sigma$ the orbit closure corresponding to the cone $\sigma$, we can define a $B^*_T$-linear map
\begin{equation}\label{GeneratorsForHomologyOfToricVariety}
\bigoplus_{\sigma \in \Delta}  \langle [V_\sigma] \rangle_{B^*_T} \to B^T_*(X_\Delta), 
\end{equation}
where $\langle [V_\sigma] \rangle_{B^*_T}$ is the free $B^*_T$-module on the symbol $[V_\sigma]$ and the morphism is defined as
$$b [V_\sigma] \mapsto b \cdot i_{\sigma *}(1_{V_\sigma}),$$
$i_\sigma$ being the closed embedding $V_\sigma \hookrightarrow X_\Delta$. By the above analysis this morphism is surjective.

We now have a nice set of generators of the equivariant homology group $B^T_*(X_\Delta)$ over the coefficient ring $B^*_T$, so we are left with the task of characterizing the relations. From the decomposition we can almost immediately conclude that the relations will be generated by those of form
\begin{equation*}
\xi' [V_\tau] = \sum_{\sigma \supset \tau} b_\sigma [V_\sigma],
\end{equation*}
where $\xi'$ is a Chern class of a line bundle associated to a character orthogonal to the fan $\tau \in \Delta$, $\sigma$ runs over all the cones of $\Delta$ containing $\tau$ and $b_i \in B^*_T$. In order to say more, we need to look more closely at the line bundles associated to linear forms.

\subsection{Line bundles associated to characters and more on the structure of $B^T_*(X_\Delta)$}\label{LinearLines}

Suppose we have a character $m = a_1 e^*_1 + \cdots + a_n e^*_n \in M$ of $T$. By Proposition \fref{ChernClassesOfCharactersProp} the first Chern class of the corresponding line bundle $V_m$ is given by
$$\xi_m := c_1(V_m) = a_1 \cdot_F \xi_1 +_F \cdots +_F a_n \cdot_F \xi_n \in B^T_{-1}(pt),$$
where $\xi_i = c_1(V_{e^*_i})$ and $F$ is the formal group law of $B_*$. We note that this defines a map from the character lattice $M$ to the topological group of the elements in 
$$\langle \xi_1,...,\xi_n \rangle \cap B^T_{-1}(pt)$$ 
with the group operation given by $+_F$. Note that if $B_i(pt)$ is trivial for $i$ negative (e.g. $B_* = \Omega_*$ is the algebraic bordism), then $\langle \xi_1,...,\xi_n \rangle \cap B^T_{-1}(pt) = B^T_{-1}(pt).$

Let us begin with an easy observation.

\begin{prop}
Let $m_1,...,m_r \in M$ be linearly independent characters. Then the $B^*$-algebra generated by $\xi_{m_i}$ inside $B^*_T$ is the free $B^*$-algebra $B^* [\xi_{m_1},...,\xi_{m_r}]$. Moreover, if $m_i$ generate $M$, then the quotient of $B^*_T$ by the ideal generated by the $\xi_{m_i}$ is naturally identified with $B^*$.
\end{prop}

\begin{proof}
Let $m_i = a_1^i e^*_1 + \cdots + a_n^i e^*_n$, where $a^i_j \in \Z$. By the properties of formal group law $F$ of the theory $B_*$, we see that 
\begin{align*}
c_1(V_{m_i}) &= a_1^i \cdot_F \xi_1 +_F \cdots +_F a_n^i \cdot_F \xi_n \\
&= a_1^i \xi_1 + \cdots + a_n^i \xi_n + \mathcal{O}(\mathrm{quadratic \ in \ \xi_i}).
\end{align*}
Using this, one can show that the algebra generated by the Chern classes in $B_*^T(pt)$ is the free algebra on $c_1(V_{m_i})$.

For the second claim, we first observe from the formal group law that $\xi_{m_i}$ is always contained in the ideal generated by $\xi_1,...,\xi_n$, and hence the natural map $B^T_* \to B_*$ descends to a map $B^T_* / (\xi_{m_1}, ..., \xi_{m_r}) \to B_*$. But as $m_i$ generate, $x_j$ can be expressed as an integral sum of $m_i$, and hence $\xi_j$ can be expressed as a formal sum of $\xi_{m_1}, ..., \xi_{m_r}$. From this it follows that $\xi_j$ is contained in the ideal $(\xi_{m_1},...,\xi_{m_r})$, and the claim follows. 
\end{proof}

We also record a statement whose proof is essentially contained in that of the Lemma \fref{EquivInj}.

\begin{prop}\label{HomologyOfOrbits}
Let $\sigma$ be a nonsingular cone having a lattice basis $v_1, ..., v_r$ in $N_\R \cong \R^n$. Then 
\begin{equation*}
B^T_*(O_\sigma) \cong B^*_T / (\xi_{m_{r+1}},...,\xi_{m_n}),
\end{equation*}
with an appropriate degree change, where $O_\sigma$ is the (non-closed) orbit corresponding to $\sigma$, and $m_i$ form an integral basis for the linear forms $m \in M$ orthogonal to $\sigma$.
\end{prop}
\begin{proof}
Extend $m_{r+1},...,m_{n}$ to an integral basis $m_1,...,m_n$ of $M$. Now the open set $\mathcal U_{\sigma}$ corresponding to the cone $\sigma$ is naturally identified with
\begin{equation*}
\Spec ( k [\chi^{m_1},...,\chi^{m_r},\chi^{\pm m_{r+1}}, ... ,\chi^{\pm m_n}]).
\end{equation*}
The torus orbit $O_\sigma \hookrightarrow \Uc_\sigma$ is the vanishing locus of $\chi^{m_1},...,\chi^{m_r}$. It splits as the product 
\begin{equation*}
\Spec ( k [\xi^{\pm m_{r+1}}]) \times \cdots \times \Spec ( k [\chi^{\pm m_n}])
\end{equation*}
of $T$-varieties, so in order to prove the claim, it is enough by the previous lemma and Proposition \fref{EquivariantChernOfRepresentations} to look at the $T$-action on $U = \Spec (k[\chi^{\pm m}])$, where $m$ is an arbitrary character $a_1 e^*_1 + \cdots + a_n e^*_n \in M$, and to make sure that it coincides with
\begin{equation*}
(\lambda_1,...,\lambda_n) u = \lambda_1^{a_1} \cdots \lambda_n^{a_n}.
\end{equation*} 
But as this action arises from the map of $k$-algebras
\begin{align*}
k[\chi^{\pm m}] &\to k[x_1^{\pm 1}, ... ,x_1^{\pm 1}] \otimes_k k[\chi^{\pm m}]\\
\chi^{m} &\mapsto x_1^{a_1} \cdots x_n^{a_n} \otimes \chi^m,
\end{align*} 
this is certainly true, and hence we are done.
\end{proof}

\subsubsection*{The elements $\xi_i$ and divisors of $X_\Delta$}

In order to have a geometric interpretation of multiplying elements of $B^T_*(X_\Delta)$ by $\xi_i$, we will express them with the help of divisors of $X_\Delta$. This is made easy by the following happy accident: if we use the good system of representations of Example \fref{StdSystemOfRepsForT}, then the intermediate approximations $U_K \times^T X_\Delta$ turn out to be toric varieties. (We replaced $i$ with $K$ to make the notation in the following less painful).

Recall that $U_K = (\A^{K+1} - 0)^n$ where the $i^{th}$ coordinate of $T$ acts diagonally on the $i^{th}$ copy of $\A^{K+1} - 0$. One immediately observes that this is a toric variety for $T^{K+1}$: it is given by the products of the fans one obtains from the standard $K+1$-cone $\langle e_0,...,e_N \rangle_{\R_{\geq 0}}$ in $\R^{K+1}$ after removing the interior of the maximal cone. Moreover, since the product of toric varieties is given by the product of respective fans, $U_K \times X_\Delta$ is a toric variety for $T^{K+2}$ given by the product fan inside the $n(K+2)$ dimensional vector space $N_\Rb^{K+2}$, whose basis we are going to denote by 
\begin{equation*}
e^1_0,...,e^1_K,e^2_0,...,e^2_K,...,e^n_0,...,e^n_K,e_1,...,e_n,
\end{equation*}
where $e^i_j$ is the $j^{th}$ basis vector for the space corresponding to the $i^{th}$ copy of $\A^{K+1}-0$, and $e_i$ are the basis vectors for the space corresponding to $X_\Delta$. 

The mixed quotient $U_K \times^T X_\Delta$ is a toric variety for the torus $T^{K+1}$ identified as the quotient $T^{K+2} / T$. Recalling that the action of $T$ on $U_K \times X_\Delta$ is the anti-diagonal one, we see that the lattice $N^{K+1}$ of one-parameter subgroups of $T^{K+1}$ corresponds to the quotient of $N^{K+2}$ by the sublattice generated by the vectors
\begin{equation*}
e_1 - (e^1_0 + ... + e^1_K), ... , e_n - (e^n_0 + ... + e^n_K). 
\end{equation*}
Note that the images of 
\begin{equation*}
e^1_0,...,e^1_{K-1},e^2_0,...,e^2_{K-1},...,e^n_0,...,e^n_{K-1},e_1,...,e_n
\end{equation*}
form a basis for the quotient space. It is now clear that the fan inside $N^{K+1}$ is \emph{almost} the fan corresponding to $\Proj^K \times ... \times \Proj^K \times X_\Delta$, except that the ''back rays'' (which correspond to the images of $e^i_K$) are not just simply $-(\overline e^i_0 + ... + \overline e^i_{K-1})$, but instead $\overline e_i-(\overline e^i_0 + ... + \overline e^i_{K-1})$. Thus the mixed quotient is an $X_\Delta$-bundle over $(\Proj^K)^n$.

From the last observation, we immediately get the following identity using the standard properties of divisors on toric varieties.

\begin{lem}
Let everything be as above, and let $D_i$ be the Cartier divisor associated to the ray generated by $\overline e_i-(\overline e^i_0 + ... + \overline e^i_{K-1})$. Then 
$$-[D_i] = \sum_{\rho \in \Delta} \langle e_i^*, v_\rho \rangle [D_\rho] \in \mathrm{Pic}(U_K \times^T X_\Delta),$$
where $\rho$ runs over all the rays of $\Delta$, $v_\rho$ is the primitive lattice vector generating $\rho$, and $D_\rho$ is the divisor associated to the ray $\rho$.
\end{lem} 
\begin{proof}
Indeed, this is exactly the relation given by the character $e_i^*$, see \cite{Ful2} Section 3.2 for more details on Picard groups of toric varieties.
\end{proof}

\begin{rem}
Note that in the above formula, when computing $\langle e_i^*, v_\rho \rangle$, it does not matter if we consider $e_i^*$ and $v_\rho$ to lie in $M^{K+2}$ and $N^{K+2}$ (with the convention for basis of $N^{K+2}$ we have fixed above), respectively, or in $M$ and $N$ respectively --- the result will be the same.
\end{rem}

This result is readily interpreted as an isomorphism of line bundles on $U_K \times^T X$
\begin{equation*}
\pi_i^* \mathcal O (-1) \cong \mathcal{O}(D_{\rho_1})^{\bigotimes \langle e_i^*, v_{\rho_1} \rangle} \otimes \cdots \otimes \mathcal{O}(D_{\rho_r})^{\bigotimes \langle e_i^*, v_{\rho_r} \rangle},
\end{equation*}
where $\rho_i$ are the rays of $\Delta$ enumerated in some order, and $\pi_i$ is the natural map 
$$U_K \times^T X \to (\Pb^K)^n \xrightarrow{\mathrm{pr}_i} \Pb^K.$$
As $\mathrm{pr}_i^* \mathcal O (-1)$ corresponds to the character line bundle of $e^*_i \in M$, we obtain the following identity by taking limits and linear combinations of the basis characters $e^*_i$.

\begin{lem}
Let everything be as above, and let $m \in M$ be a character of $T$. Then we get an equality of equivariant Chern class operators
$$\xi_m = \langle m, v_{\rho_1} \rangle \cdot_F c_1(D_{\rho_1}) +_F \cdots +_F \langle m, v_{\rho_r} \rangle \cdot_F c_1(D_{\rho_r}): B_*^T(X_\Delta) \to B_{*-1}^T(X_\Delta),$$
where $c_1(D_{\rho_i})$ is a shorthand for $c_1(\mathcal O (D_{\rho_i}))$.
\end{lem}

Consider the generators for $B^T_*(X_\Delta)$, where $X_\Delta$ is a non-singular toric variety for $T$, given by the surjective morphism (\fref{GeneratorsForHomologyOfToricVariety}). Let $V_\tau$ be an orbit closure corresponding to a cone $\tau \in \Delta$, and let $m$ be a character orthogonal to $\tau$. Then the above Lemma gives us the relation
\begin{equation*}
\xi_m [V_\tau] = (\langle m, v_{\rho_1} \rangle \cdot_F c_1(D_{\rho_1}) +_F \cdots +_F \langle m, v_{\rho_r} \rangle \cdot_F c_1(D_{\rho_r})) \cap [V_\tau].
\end{equation*}
Notice that by assumption $\langle m, v_{\rho} \rangle = 0$ for all rays $\rho$ in $\tau$. Because equivariant closed inclusions to smooth toric varieties induce injective pushforward morphisms in $B^T_*$, relations of above form are the only relations in $B^T_*$ between the generators $[V_\sigma]$. However, it seems hard to transform the above formulas to a useful form: one should replace all terms containing self intersections with $B^T_*$-linear combinations of $[V_\sigma]$, and this seems hard to do in general. 
 
\begin{rem}\label{ChowOfToric}
One can easily read off the description of the Chow groups of toric varieties achieved in \cite{FS} from this description, although one gets it immediately only for smooth toric varieties. As the formal group law of $\ch_*$ is the additive one, the relations we get for the equivariant Chow groups are
\begin{align*}
\xi_m [V_\tau] &= (\langle m, v_{\rho_1} \rangle c_1(D_{\rho_1}) + \cdots + \langle m, v_{\rho_r} \rangle c_1(D_{\rho_r})) \cap [V_\tau] \\ 
&= \sum_{\sigma \supset \tau} \langle m, n_\sigma \rangle [V_\sigma],
\end{align*}
where $\sigma$ runs over all the cones in $\Delta$ of one dimension higher than $\tau$ containing $\tau$, and $n_\sigma$ is the primitive generator for the ray of $\sigma$ not in $\tau$. Passing to the non-equivariant case (by setting all $\xi_m = 0$), we obtain
\begin{equation*}
\sum_{\sigma \supset \tau} \langle m, n_\sigma \rangle [V_\sigma] = 0
\end{equation*}
exactly corresponding to the relations given in \cite{FS}. We note, however, that here we fully described the $T$-equivariant Chow groups as well, at least for smooth toric varieties. The presentation in the singular case should follow easily from the descent exact sequence, and the proof left as an exercise for the reader.
\end{rem}

\subsection{Application --- algebraic bordism groups $\Omega_*(X_\Delta)$ of some toric varieties}\label{Computations}

In this subsection we apply the results obtained above in order to show through examples that it is sometimes possible to fully determine the structure of $\Omega_*(X_\Delta)$ where $X_\Delta$ is a singular toric variety. The main idea is to take a toric resolution $X_{\widetilde{\Delta}} \to X_\Delta$ of $X_\Delta$, compute the algebraic bordism of the smooth toric variety $X_{\widetilde{\Delta}}$, and then use
geometric arguments to determine $\Omega_*(X_\Delta)$. We will start with a general proposition, showing that oriented Borel--Moore homology theories satisfying localization and descent also satisfy a certain cosheaf condition.

\begin{prop}\label{CosheafCondition}
Let $B_*$ be an oriented Borel--Moore homology theory satisfying localization and descent. Moreover, let 
\begin{center}
\begin{tikzpicture}[scale=2]
\node (A1) at (0,1) {$E$};
\node (B1) at (1,1) {$\wtil X$};
\node (A2) at (0,0) {$Z$};
\node (B2) at (1,0) {$X$};

\path[every node/.style={font=\sffamily\small}]
(A1) edge[->] node[right]{$\pi'$} (A2)
(B1) edge[->] node[right]{$\pi$} (B2)
(A1) edge[->] node[above]{$i'$} (B1)
(A2) edge[->] node[above]{$i$} (B2)
;
\end{tikzpicture}
\end{center}
be an abstract blow-up square with $\pi$ an envelope. Then the sequence
$$B_*(E) \xrightarrow{(-\pi'_*, i'_*)} B_*(Z) \oplus B_*(\wtil X) \xrightarrow{i_* + \pi_*} B_*(X) \to 0$$
is exact.
\end{prop}
\begin{rem}
Recall that being an abstract blow-up square means that $i$ is a closed embedding, that $\pi$ is proper, birational, and an isomorphism over the complement of $Z$, and that the square is Cartesian. 
\end{rem}
\begin{proof}
Being a proper envelope is clearly stable under pullbacks, so $\pi'$ is a proper envelope. Moreover, since $B_*$ satisfies descent, $\pi_*$ is surjective, and therefore the complex is exact on the right. We are left to show the exactness at the middle.

We start by considering the commutative diagram
\begin{center}
\begin{tikzpicture}[scale=2]
\node (A1) at (0,1) {};
\node (A2) at (0,0) {$0$};
\node (B1) at (1,1) {$B_*(E \times_Z E)$};
\node (B2) at (1,0) {$B_*(E)/K$};
\node (C1) at (3,1) {$B_*(\wtil X \times_X \wtil X)$};
\node (C2) at (3,0) {$B_*(\wtil X)$};
\node (D1) at (5,1) {$B_*(U)$};
\node (D2) at (5,0) {$B_*(U)$};
\node (E1) at (6,1) {$0$};
\node (E2) at (6,0) {$0$};

\path[every node/.style={font=\sffamily\small}]
(B1) edge[->] node[right]{$\mathrm{pr}_{1*} - \mathrm{pr}_{2*}$} (B2)
(C1) edge[->] node[right]{$\mathrm{pr}_{1*} - \mathrm{pr}_{2*}$} (C2)
(D1) edge[->] node[right]{$0$} (D2)
(A2) edge[->] node[below]{} (B2)
(B2) edge[->] node[below]{$i'_*$} (C2)
(C2) edge[->] node[below]{$j'^*$} (D2)
(D2) edge[->] node[below]{} (E2)
(B1) edge[->] node[above]{$(i' \times_X i')_*$} (C1)
(C1) edge[->] node[above]{$(j' \times_X j')^*$} (D1)
(D1) edge[->] node[below]{} (E1)
;
\end{tikzpicture}
\end{center}
where $j'$ is the inclusion of the open complement of $Z$ to $\wtil X$, and $K$ is the kernel of $i'_*$. Note that the rows of this diagram are essentially localization exact sequences, and the columns are essentially the beginning of the descent exact sequence, the only difference being that we have $B_*(E) /K$ instead of just $B_*(E)$. By applying the snake lemma, we get an exact sequence
$$B_*(U) \xrightarrow{\delta} B_*(Z) / K' \xrightarrow{i_*} B_*(X) \xrightarrow{j^*} B_*(U) \to 0,$$
where $K'$ is the image of $K$ under $\pi'_*$. We claim that the connecting morphism $\delta$ is 0: indeed, by construction, $\delta(u)$ is obtained by first lifting $u$ to $\wtil u \in B_*(\wtil X \times_X \wtil X)$, then pushing it down to $B_*(\wtil X)$, lifting again to a class in $B_*(E)/K$, and then finally pushing down to $B_*(Z)/K'$. But one way of doing this is first finding a lift $\wtil u' \in B_*(\wtil X)$, and then setting
$$\wtil u := \Delta_{\wtil X / X *} (\wtil u ') \in B_*(\wtil X \times_X \wtil X),$$
where $\Delta_{\wtil X / X}$ is the diagonal morphism. As $\mathrm{pr}_{1*}(\wtil u) - \mathrm{pr}_{2*}(\wtil u)= 0$, we see that $\delta (u) = 0$, and therefore $i_*: B_*(Z)/K' \to B_*(X)$ is injective.

The above arguments have shown us that the rows of the commutative diagram
\begin{center}
\begin{tikzpicture}[scale=2]
\node (A1) at (0,1) {$K$};
\node (A2) at (0,0) {$K$};
\node (B1) at (1,1) {$B_*(E)$};
\node (B2) at (1,0) {$B_*(Z)$};
\node (C1) at (2.5,1) {$B_*(\wtil X)$};
\node (C2) at (2.5,0) {$B_*(X)$};
\node (D1) at (4,1) {$B_*(U)$};
\node (D2) at (4,0) {$B_*(U)$};
\node (E1) at (5,1) {$0$};
\node (E2) at (5,0) {$0$};

\path[every node/.style={font=\sffamily\small}]
(A1) edge[->] node[right]{$\mathrm{Id}$} (A2)
(B1) edge[->] node[right]{$\pi'_*$} (B2)
(C1) edge[->] node[right]{$\pi_*$} (C2)
(D1) edge[->] node[right]{$\mathrm{Id}$} (D2)
(A2) edge[->] node[below]{} (B2)
(B2) edge[->] node[below]{$i_*$} (C2)
(C2) edge[->] node[below]{$j^*$} (D2)
(D2) edge[->] node[below]{} (E2)
(A1) edge[->] node[below]{} (B1)
(B1) edge[->] node[above]{$i'_*$} (C1)
(C1) edge[->] node[above]{$j'^*$} (D1)
(D1) edge[->] node[below]{} (E1)
;
\end{tikzpicture}
\end{center}
are exact. One then shows by a simple diagram chase, that given $a \in B_*(\wtil X)$ and $b \in B_*(Z)$ whose images in $B_*(X)$ coincide, there exists $c \in B_*(E)$ so that $i'_*(c) = a$ and $\pi'_*(c)=b$. This proves the exactness of
$$B_*(E) \xrightarrow{(-\pi'_*, i'_*)} B_*(Z) \oplus B_*(\wtil X) \xrightarrow{i_* + \pi_*} B_*(X),$$
so we are done.  
\end{proof}

\begin{ex}
Consider the complete fan $\Delta$ spanned by the rays $\tau_1 = \langle 1,0 \rangle$, $p_n = \langle -1, n \rangle$, $\tau_3 = \langle -1, 0 \rangle$ and $q_m = \langle -1, -m \rangle$, $m,n \geq 1$. We obtain a resolution $\widetilde \Delta$ by adding the rays $\tau_2 = \langle 0,1 \rangle$, $\tau_4 = \langle 0,-1 \rangle$, $p_i = \langle -1, i \rangle$ and $q_j = \langle -1, -j \rangle$ for $i = 1..n-1$ and $j = 1..m$. 

In order to compute $\Omega_*(X_{\widetilde \Delta})$, we first note that the relations generated by the rays of the fan simply identify all the maximal cones with each other. Hence we are left with the presentation
\begin{align*}
\Omega_*(X_{\widetilde \Delta}) = \langle &s, \tau_1, \tau_2, \tau_3, \tau_4, p_1,...,p_{n}, q_1,..., q_{m}, \sigma \mid \\
& \tau_1 -_F (\tau_3 +_F p_1 +_F + \cdots +_F p_n +_F q_1 +_F \cdots +_F q_m) = 0, \\
& p_1 +_F 2 \cdot_F p_2 +_F \cdots +_F n \cdot_F p_n -_F (q_1 +_F 2 \cdot_F q_2 +_F \cdots +_F m \cdot_F q_M) =0 \rangle.
\end{align*}
In order to arrive at $\Omega_*(X_\Delta)$, we notice that all the exceptional divisors are isomorphic to either $\Pb^1$ or chains of $\Pb^1$, and it is easy to argue using Proposition \fref{CosheafCondition} that the additional relations we must add are exactly
\begin{equation*}
\tau_2 = \tau_4 = p_1 = \cdots = p_{n-1} = q_1 = \cdots = q_{m-1} = [\Proj^1] \times pt = -a_{11} \sigma.
\end{equation*}
As
\begin{align*}
p_1 +_F 2 \cdot_F p_2 +_F \cdots +_F n \cdot_F p_n &= \sum_{i=1}^n i p_i - a_{11} \Bigg(\sum_{i=1}^{n-1} i (i-1) - {n(n-1) \over 2} + \sum_{i=1}^{n-1} i(i+1)\Bigg) \sigma \\
&= \sum_{i=1}^n i p_i + a_{11}{n(n-1) \over 2} \sigma \\
&= n p_n,
\end{align*}
(first linear terms, then two terms from self intersections, and finally term coming from intersections between consecutive $p_i$) and as $p_i$ and $q_j$ do not intersect, we finally arrive at the description
\begin{align*}
\Omega_*(X_\Delta) =  \langle &s, \tau_3, p_{n}, q_{m}, \sigma \mid np_n - m q_m = 0 \rangle.
\end{align*}
\end{ex}

\begin{ex}\label{AlgCobOfCube}
Consider next $\Delta$ to be the fan over the cube, as in the end of Chapter 2 of \cite{Ful2}. The cube has vertices at points $(\pm 1, \pm 1, \pm 1)$ and we consider this as a rational polytope in the lattice generated by the vertices of the cube. The fan is the fan whose cones are generated by the faces, edges and vertices of the cube. As a toric resolution $X_{\widetilde \Delta} \to X_{\Delta}$ we subdivide each face of the cube diagonally, in order to obtain a fan corresponding to $\Proj^3$ blown up at the four $T$-fixed points.

As $X_{\widetilde \Delta}$ is obtained by blowing up $4$-points on the $\Proj^3$, it follows that
\begin{equation*}
\Omega_*(X_{\widetilde \Delta}) = \langle s, \tau, \sigma, \rho, \tau_1, \sigma_1, ...,\tau_4,  \sigma_4 \rangle,
\end{equation*} 
where the first four basis elements come from the $\Proj^3$ before blowup, and the next eight are the new elements introduced by the four blowups. As the original variety $X_\Delta$ can be obtained from $X_{\wtil \Delta}$ by contracting the strict transforms of the six lines connecting $T$-equivariant points of $\Pb^3$, and as the class of the strict transform of the line connecting $i^{th}$ and $j^{th}$ $T$-fixed point has to be of form
$$\sigma - \sigma_i - \sigma_j + b \rho$$
for some $b \in \Omega_1(pt)$ not depending on $i$ and $j$, we need to add the relations
$$\sigma - \sigma_i - \sigma_j = - (a_{11} + b) \rho \qquad (i \not = j)$$
to $\Omega_*(X_{\wtil \Delta})$ in order to obtain $\Omega_*(X_\Delta)$ (another easy application of Proposition \fref{CosheafCondition}). These relations have the effect of identifying the classes of $\sigma_i$ with each other --- let us denote this class by $\sigma'$ --- and finally $\sigma = 2 \sigma' - (a_{11} + b) \rho$. Hence,
\begin{equation*}
\Omega_*(X_{\Delta}) = \langle s, \tau, \tau_1, \tau_2, \tau_3, \tau_4,  \sigma', \rho \rangle.
\end{equation*}

\end{ex}

\begin{rem}
It is worth to note that the algebraic bordism groups of the singular toric varieties we computed in these two examples are very similar to the corresponding Chow groups. In fact, one immediately observes that they are abstractly isomorphic to the tensor product of the Chow groups with the Lazard ring. One could ask the question if this is always the case. The author was not able to find any counter examples.
\end{rem}

\subsection{Stanley--Reisner presentations and piecewise functions on a fan}\label{Generalities}

Let us take another look at the relations of $B^T_* (X_\Delta)$ for a nonsingular toric variety $X_\Delta$ found in the end of the Section \fref{LinearLines}. If we consider the equivariant homology group as the equivariant cohomology ring $B_T^* (X_\Delta) = B_{n-*}^T (X_\Delta)$, where multiplication is given by the intersection product, then the relations 
\begin{equation}\label{LinRelExample}
\xi_i = \langle e^*_i, v_{\rho_1} \rangle \cdot_F \rho_1 +_F \cdots +_F \langle e^*_i, v_{\rho_r} \rangle \cdot_F \rho_r,
\end{equation}
where $\rho_j$ are the rays of $\Delta$ enumerated in some order, transform into the following result.

\begin{thm}
Let $X_\Delta$ be a non-singular toric variety for $T$. Then the ring $B_T^* (X_\Delta)$ is isomorphic to the graded power series ring $B^* [[\rho_1, ..., \rho_r]]_\gr$ modulo the Stanley--Reisner relations
\begin{equation*}
\rho_{i_1} \cdots \rho_{i_j} = 0 
\end{equation*}
whenever $\rho_{i_1}, ..., \rho_{i_j}$ do not span a cone in $\Delta$.
\end{thm}
\begin{proof}
Certainly the Stanley--Reisner ring $R$ maps to $B_T^*(X_\Delta)$. In order to prove that this is an isomorphism, let us consider a monomial of form
\begin{equation*}
\rho_{i_1}^{n_{i_1}} \cdots \rho_{i_j}^{n_{i_j}},
\end{equation*}
where the exponents are strictly positive, and the rays span a cone $\sigma$, which we are going to call the \emph{support} of the monomial.

Suppose then that we have a nonzero element of the graded Stanley--Reisner ring. We are going to show that it determines a nonzero element in $B_T^*(X_\Delta)$. Without loss of generality, we may assume that the power series is homogeneous, i.e., it is of form
\begin{equation*}
\beta = \sum_{i_1,..., i_n} b_{i_1, ..., i_n} \rho_1^{i_1} \cdots \rho_r^{i_n}
\end{equation*}
where $b_{i_1, ..., i_r} \in B_{k + i_1 + \cdots + i_r}(pt)$ for some fixed $k$. Let $\sigma$ be a cone that is the support of a nontrivial term of $\beta$, and let $j_\sigma: O_\sigma \hookrightarrow X_\Delta$ be the locally closed inclusion of the orbit corresponding to $\sigma$. We claim that
$$j_\sigma^*(\beta) \not = 0 \in B_T^*(O_\sigma).$$
Indeed, without a loss of generality, $\sigma$ is the cone spanned by $e_1,...,e_r$ and hence by Proposition \fref{HomologyOfOrbits}
$$B_T^*(O_\sigma) \cong B^*(pt)[[\xi_1,...,\xi_n]]_\gr/(\xi_{r+1},...,\xi_{n}).$$
Moreover, if $\rho_i$, where $i \leq r$, is a ray of $\sigma$, then we can use the relations (\fref{LinRelExample}) to conclude that
$$\xi_i = \rho_i \in B_T^*(O_\sigma),$$
and therefore $j_\sigma^*(\beta) \not = 0 \in B_T^*(O_\sigma)$, proving the injectivity of the map $R \to B_T^*(X)$. The proof of surjectivity is of the same spirit, and is left to the reader.
\end{proof}

The above result extends the well known results known for Chow groups, $K$-theory and algebraic bordism. It also illuminates that in all of the cases the reason of the Stanley--Reisner ring appearing is the same.

Another well known presentation of the $T$-equivariant cohomology ring of a smooth toric variety $X_\Delta$ is in terms of the global sections of a sheaf of functions (of some kind) on the fan $\Delta$. Examples include the $T$-equivariant Chow ring, which is described in terms of piecewise polynomial functions on $\Delta$, and the $T$-equivariant algebraic cobordism, where we have piecewise graded power series on $\Delta$. This presentation has the advantage of generality over the Stanley--Reisner presentation: usually the same presentation describes the $T$-equivariant \emph{operational} cohomology rings of singular toric varieties as well (see \cite{AP,Pay,Karu}). We can obtain similar description for arbitrary theory $B_*^T$, at least in the smooth case.

We quickly recall what we mean by functions on a fan. Consider a smooth toric variety $X_\Delta$. Since $X_\Delta$ is smooth, the inclusions of the orbits $O_\sigma \to X_\Delta$ are regular, and therefore we get a natural map
\begin{equation*}
i: B_T^*(X) \to \prod_{\sigma \in \Delta}B_T^*(O_\sigma)
\end{equation*}
induced by the l.c.i. pullbacks. We can think $B_T^*(O_\sigma)$ as a stalk of a sheaf at the cone $\sigma$, and by the description $B_T^*(O_\sigma)$ we obtained earlier, we see that for a inclusion $\tau \subset \sigma$ of fans, we get a surjective restriction morphism
\begin{equation*}
B^*_T(O_\sigma) \to B^*_T(O_\tau)
\end{equation*}
and the basic functoriality properties of these restriction morphisms imply that the rings at cones glue together to give a sheaf on the fan, where the open sets are taken to be the subfans. We call this the sheaf of graded power series on the fan $\Delta$.

\begin{thm}
Let everything be as above. Then the map $i$ identifies $B^*_T(X_\Delta)$ with the global sections of the sheaf of graded power series on $\Delta$. 
\end{thm}
\begin{proof}
The proof is mostly formal from the Stanley--Reisner description. Consider at first a single l.c.i. pullback
\begin{equation*}
B_T^*(X_\Delta) \to B_T^*(O_\sigma)
\end{equation*}
which, we recall, is a ring homomorphism. Let $\sigma$ be spanned by rays $\tau_{i_1},...,\tau_{i_j}$, and note that $B_T^*(O_\sigma)$ can be identified as the graded power series algebra of $\tau_1,...,\tau_r$ over the base ring $B^*$. I claim that the pullback of any monomial including any other ray $\tau$ must be zero. This is because the support of such a monomial does not meet the orbit $O_\sigma$ (although it might meet its closure $V_\sigma$). This proves the injectivity of $i$, and in fact, we see that the pullback to \emph{minimal} orbits would have been injective already.

Let us then have an element $(f_\sigma)_{\sigma \in \Delta}$, where  $f_\sigma \in B^*_T(O_\sigma)$, which corresponds to a global section of the sheaf of graded power series, i.e., this collection respects the restriction maps. It is clear that we can always find an element of the Stanley--Reisner ring pulling back to this collection, finishing the proof.
\end{proof}

This description should easily generalize to the operational cohomology of an arbitrary toric variety using the techniques of \cite{Pay}, and perhaps assuming some extra compatibility conditions on the theory.

\section{Künneth formula and a universal coefficient theorem}\label{StructureResults}

In this section we generalize the results of \cite{FS} concerning Künneth formula and operational cohomology rings to arbitrary Borel--Moore homology theories satisfying certain extra assumptions. This will be fairly straightforward after all the work in the previous section. Throughout the section $B_*$ will denote an oriented Borel--Moore homology theory satisfying localization and $T$ will be a fixed split torus of rank $n$.

\subsection{Künneth Formula}\label{KunnethSec}

The purpose of this subsection is to prove the following:

\begin{thm}\label{Kunneth}
Let $B_*$ be an oriented Borel--Moore homology theory satisfying localization and descent, and let $X_\Delta$ be a toric variety. Then the Künneth morphism
\begin{equation*}
B_*(X_\Delta) \otimes_{B^*} B_*(Y) \to B_*(X_\Delta \times Y)
\end{equation*}
is an isomorphism for all varieties $Y$. (We say that $X_\Delta \times Y$ satisfies Künneth formula in $B_*$.)
\end{thm}

The proof is based to reducing this question to the $T$-equivariant case, when $X_\Delta$ is a toric variety for $T$. We begin with the case $X_\Delta$ nonsingular.

\begin{lem}\label{TKunnethLem}
Let $X_\Delta$ be a smooth toric variety for $T$, and let $Y$ be a variety with a trivial $T$-action. Then $X_\Delta \times Y$ satisfies Künneth formula in $B^T_*$.
\end{lem}
\begin{proof}
Suppose $O$ is a torus orbit in $X_\Delta$. Then, by Proposition \fref{EquivariantChernOfRepresentations} we know that the Künneth morphism
$$B^T_*(O) \otimes_{B_T^*} B^T_*(Y) \to B^T_*(O \times Y)$$
is an isomorphism. Moreover, if $\sigma$ is a maximal cone of $\Delta$, which we can assume without loss of generality to be generated by the standard lattice vectors $e_1,...,e_r$, then we can combine the above with Proposition \fref{TrivialAction} to conclude that 
\begin{align*}
B^T_*(O_\sigma \times Y) & \cong B^T_*(O_\sigma) \otimes_{B^*_T} B^T_* (Y) \\
&\cong \big(B^*_T / (\xi_{r+1},...,\xi_{n}) \big) \otimes_{B^*_T} \big( B^*_T \widehat \otimes_{B^*} B_*(Y) \big) \\
&\cong B^* [[\xi_1,...,\xi_r]]_\gr \widehat \otimes_{B^*} B_*(Y).
\end{align*}
If we denote by $i$ the closed embedding $O_\sigma \times Y \hookrightarrow X_\Delta \times Y$, then we can deduce as in the proof of Lemma \fref{EquivInj} that $i^*i_*$ is just multiplication by $\xi_1 \cdots \xi_r$, and therefore the pushforward $i_*$ is injective.

We can now combine the above investigation with localization exact sequences to obtain the commutative diagram
\begin{center}
\begin{tikzpicture}[scale=2]
\node (00) at (-1.5,0) {$0$};
\node (A1) at (0,1) {$B_*^T(O_\sigma) \otimes_{B^*_T} B_*^T(Y)$};
\node (B1) at (2.2,1) {$B_*^T(X_\Delta) \otimes_{B^*_T} B_*^T(Y)$};
\node (C1) at (4.4,1) {$B_*^T(X_{\Delta'}) \otimes_{B^*_T} B_*^T(Y)$};
\node (D1) at (5.8,1) {$0$};
\node (A2) at (0,0) {$B_*^T(O_\sigma \times Y)$};
\node (B2) at (2.2,0) {$B_*^T(X_\Delta \times Y)$};
\node (C2) at (4.4,0) {$B_*^T(X_{\Delta'} \times Y)$};
\node (D2) at (5.8,0) {$0$};

\path[every node/.style={font=\sffamily\small}]
(A1) edge[->] node[right]{$\cong$} (A2)
(B1) edge[->] (B2)
(C1) edge[->] node[right]{$\cong$} (C2)
(A1) edge[->] (B1)
(B1) edge[->] (C1)
(C1) edge[->] (D1)
(A2) edge[->] (B2)
(B2) edge[->] (C2)
(C2) edge[->] (D2)
(00) edge[->] (A2)
;
\end{tikzpicture}
\end{center}
with exact rows, where the rightmost vertical map can be assumed to be an isomorphism by induction on the number of cones in $\Delta$ (the base case of an empty fan being Corollary \fref{KunnethForT}). It follows from $5$-lemma that also the middle vertical arrow is an isomorphism.
\end{proof}

To generalize the previous result to singular varieties, we need to assume that the theory $B_*$ satisfies descent (this holds for $\Omega_*$ by \cite{Karu2}). Recall that if $B_*$ satisfies descent, then so does the equivariant version $B_*^T$.

\begin{lem}\label{TKunnethLem2}
Let $B_*$ be as above, $X_\Delta$ an arbitrary toric variety, and $Y$ a variety with a trivial $T$-action. Then $X_\Delta \times Y$ satisfies Künneth formula in $B^T_*$. 
\end{lem}
\begin{proof}
Pick a toric resolution $X_{\widetilde \Delta}$ of $X_\Delta$, and denote $X = X_\Delta$, $\widetilde X = X_{\widetilde \Delta}$. By descent assumption, we have the commutative diagram
\begin{center}
\begin{tikzpicture}[scale=2]
\node (A1) at (0,1) {$B_*^T(\widetilde X \times_X \widetilde X) \otimes_{B^*_T} B_*^T(Y)$};
\node (B1) at (2.5,1) {$B_*^T(\widetilde X) \otimes_{B^*_T} B_*^T(Y)$};
\node (C1) at (5,1) {$B_*^T(X) \otimes_{B^*_T} B_*^T(Y)$};
\node (D1) at (6.5,1) {$0$};
\node (A2) at (0,0) {$B_*^T((\widetilde X \times_X \widetilde X) \times Y)$};
\node (B2) at (2.5,0) {$B_*^T(\widetilde X \times Y)$};
\node (C2) at (5,0) {$B_*^T(X \times Y)$};
\node (D2) at (6.5,0) {$0$};

\path[every node/.style={font=\sffamily\small}]
(A1) edge[->] (A2)
(B1) edge[->] node[right]{$\cong$} (B2)
(C1) edge[->] (C2)
(A1) edge[->] (B1)
(B1) edge[->] (C1)
(C1) edge[->] (D1)
(A2) edge[->] (B2)
(B2) edge[->] (C2)
(C2) edge[->] (D2)
;
\end{tikzpicture}
\end{center}
with exact rows, where the middle vertical map is known to be isomorphism by the previous lemma. In order to show that the rightmost vertical map is an isomorphism, it is enough to show that the leftmost vertical map is a surjection.

As $\widetilde X$ and $X$ are toric varieties and the map $\widetilde X \to X$ equivariant, we see that $\widetilde X \times_X \widetilde X$ has a filtration by tori. Indeed, if $\widetilde O$ is an orbit in $\widetilde X$ and $O$ is its image in $X$, the restricted map $\widetilde O \to O$ is essentially just the projection $(\alpha_1, ..., \alpha_r) \mapsto (\alpha_1, ..., \alpha_s)$, and hence $\widetilde{O} \times_O \widetilde{O}$ is isomorphic to a torus (although this is no longer necessarily a single $T$ orbit). Thus $\widetilde X \times_X \widetilde X$ is a $T$-linear scheme. We know that a product with a linear variety has a surjective Künneth map by Proposition \fref{ModifiedTotaro}, so we are done. 
\end{proof}

Bringing the theorem back to the ordinary case is now easy.
\begin{proof}[Proof of Theorem \fref{Kunneth}.] Let $X_\Delta$ be a toric variety with torus $T$, and let $Y$ be an arbitrary variety. Now we know that the equivariant Künneth map 
\begin{equation*}
B_*^T(X_\Delta) \otimes_{B^*_T} B_*^T(Y) \to B_*^T(X_\Delta \times Y)
\end{equation*}
is an isomorphism. This isomorphism is preserved after tensoring both sides with $B^*$ considered as a $B_T^*$-algebra in the natural way. On the right hand side, this tensor product equals $B_*(X_\Delta \times Y)$, and on the left hand side, we get
\begin{align*}
B^* \otimes_{B_T^*} (B_*^T(X_\Delta) \otimes_{B^*_T} B_*^T(Y)) &=  (B^* \otimes_{B_T^*} B_*^T(X_\Delta)) \otimes_{B^* \otimes_{B_T^*} B^*_T} (B^* \otimes_{B_T^*} B_*^T(Y)) \\
&= B_*(X_\Delta) \otimes_{B^*} B_*(Y).
\end{align*}
This proves the claim.
\end{proof}

\begin{rem}
Note how Lemma \fref{TKunnethLem} gives a $T$-equivariant Künneth isomorphism 
\begin{equation*}
B_*^T(X_\Delta) \otimes_{B^*_T} B_*^T(Y) = B_*^T(X_\Delta \times Y)
\end{equation*}
only if $Y$ has a trivial $T$-action. We can do better, as the following result illustrates (consider the case when $G = T$ and $T \times T$ acts on $X_\Delta$ codiagonally).
\end{rem}
\begin{thm}\label{EquivariantKunneth}
Let $B_*$ be an oriented Borel--Moore homology theory satisfying descent. Let $G$ be a linear algebraic group, and suppose that a toric variety $X_\Delta$ has a $T \times G$-action extending the $T$-action. Then, for any $G$-scheme $Y$, the Künneth morphism 
$$B_*^G(X_\Delta) \otimes_{B^*_G} B_*^G(Y) \to B_*^G(X_\Delta \times Y)$$
is an isomorphism.
\end{thm}
\begin{proof}
Recall that by Proposition \fref{GTEquivariant}, the $T$-equivariant version of $B^G_*$ is well-defined. Extend the $G$-action on $Y$ to a $T \times G$-action by letting $T$ act trivially. The arguments proving Lemmas \fref{TKunnethLem} and \fref{TKunnethLem2} also show that the Künneth morphism
$$(B^G)^T_*(X_\Delta) \otimes_{(B_G)_T^*} (B^G)^T_*(Y) \to (B^G)^T_*(X_\Delta \times Y)$$
is an isomorphism. One gets the desired result for $B^G_*$ with the same argument as in the proof of Theorem \fref{Kunneth}.
\end{proof}

\subsection{Universal coefficient theorem for operational cohomology}\label{UnivCoeffSec}

We are now ready to prove the universal coefficient theorem. Throughout this subsection, $B_*$ is going to denote a ROBM-homology theory, i.e., an oriented Borel--Moore homology theory with refined l.c.i. pullbacks (see \cite{Karu} for more details). Again, $\Omega_*$ is an example of such a theory, as is proved in \cite{LM} and \cite{Le}. For any such theory, we may construct the \emph{operational} bivariant group, and as a special case we get the operational cohomology theory $\op B^*$. The main purpose of this section is to prove the following theorem:

\begin{thm}\label{UnivCoeff}
Let $B_*$ be a ROBM-homology theory, and let $X$ be a proper variety having the property that the Künneth formula holds for $X \times Y$ for all $Y$. Then there is a natural isomorphism
\begin{equation*}
\op B^*(X) \cong \Hom_{B^*} (B_*(X), B^*).
\end{equation*}
\end{thm}

Combining this theorem with the Künneth isomorphism results of Section \fref{KunnethSec}, we get the following corollary.

\begin{cor}
Let $B_*$ be a ROBM-homology theory satisfying localization and descent, and let $X_\Delta$ be a complete toric variety for $T$. Then there is a natural isomorphism
\begin{equation*}
\op B^*(X) \cong \Hom_{B^*} (B_*(X), B^*).
\end{equation*}
Moreover, if we choose a $T \times G$-action on $X_\Delta$ extending the $T$-action, where $G$ is a linear algebraic group, then there is a natural isomorphism
\begin{equation*}
\op B^*_G(X) \cong \Hom_{B^*_G} (B^G_*(X), B^*_G).
\end{equation*}
\end{cor}

The proof of Theorem \fref{UnivCoeff} is formally the same as proof of the analogous result in \cite{FMSS}. Note that the $B^*$-module $\Hom_{B^*} (B_*(X), B^*)$ has a natural grading as the $\Hom$-module of graded modules over a graded ring. In order to have grading that coincides with the usual cohomological grading, we set $\Hom_{B^*}^k (B_*(X), B^*)$ to consist of degree preserving $B^*$-linear morphisms $B_{*+k}(X) \to B_*(pt)$. Before embarking on the proof, we quickly review the definition of operational cohomology groups.

\subsubsection*{Review on operational cohomology groups}

Here we recall the construction of operational cohomology groups. Let $X$ be any variety. An \emph{operational cohomology class} $c \in \op B^*(X)$ consists of morphisms
\begin{equation*}
c \equiv c_{Y \to X}: B_*(Y) \to B_*(Y)
\end{equation*}
for any morphism $Y \to X$. Moreover, these maps are required to satisfy the following compatibility axioms:

($C_1$) Given maps $Y' \stackrel f \to Y \to X$, where $f$ is proper, the diagram
\begin{center}
\begin{tikzpicture}[scale=2]
\node (A1) at (0,1) {$B_*(Y')$};
\node (B1) at (1,1) {$B_*(Y')$};
\node (A2) at (0,0) {$B_*(Y)$};
\node (B2) at (1,0) {$B_*(Y)$};

\path[every node/.style={font=\sffamily\small}]
(A1) edge[->] node[right]{$f_*$} (A2)
(B1) edge[->] node[right]{$f_*$} (B2)
(A1) edge[->] node[above]{$c$} (B1)
(A2) edge[->] node[above]{$c$} (B2)
;
\end{tikzpicture}
\end{center}
commutes, i.e., operational classes commute with proper pushforward.

($C_2$) Given maps $Y' \stackrel f \to Y \to X$, where $f$ is smooth, the diagram
\begin{center}
\begin{tikzpicture}[scale=2]
\node (A1) at (0,1) {$B_*(Y)$};
\node (B1) at (1,1) {$B_*(Y)$};
\node (A2) at (0,0) {$B_*(Y')$};
\node (B2) at (1,0) {$B_*(Y')$};

\path[every node/.style={font=\sffamily\small}]
(A1) edge[->] node[right]{$f^*$} (A2)
(B1) edge[->] node[right]{$f^*$} (B2)
(A1) edge[->] node[above]{$c$} (B1)
(A2) edge[->] node[above]{$c$} (B2)
;
\end{tikzpicture}
\end{center}
commutes, i.e., operational classes commute with smooth pullbacks.

($C_3$) If we have morphisms $Y \to X$, and $Y \to Z$, and an l.c.i. map $i: Z' \to Z$ inducing a Cartesian square
\begin{center}
\begin{tikzpicture}[scale=2]
\node (A1) at (0,1) {$Y'$};
\node (B1) at (1,1) {$Z'$};
\node (A2) at (0,0) {$Y$};
\node (B2) at (1,0) {$Z$};

\path[every node/.style={font=\sffamily\small}]
(A1) edge[->] (A2)
(B1) edge[->] node[right]{$i$} (B2)
(A1) edge[->] (B1)
(A2) edge[->] (B2)
;
\end{tikzpicture}
\end{center}
then the induced diagram
\begin{center}
\begin{tikzpicture}[scale=2]
\node (A1) at (0,1) {$B_*(Y)$};
\node (B1) at (1,1) {$B_*(Y)$};
\node (A2) at (0,0) {$B_*(Y')$};
\node (B2) at (1,0) {$B_*(Y')$};

\path[every node/.style={font=\sffamily\small}]
(A1) edge[->] node[right]{$i^!$} (A2)
(B1) edge[->] node[right]{$i^!$} (B2)
(A1) edge[->] node[above]{$c$} (B1)
(A2) edge[->] node[above]{$c$} (B2)
;
\end{tikzpicture}
\end{center}
commutes, i.e., operational classes commute with refined l.c.i. pullbacks.

($C_4$) If we have maps $Y \times Z \to Y \to X$, where the first map is the canonical projection, then 
\begin{equation*}
c(\alpha \times \beta) = c(\alpha) \times \beta
\end{equation*} 
in $B_*(Y \times Z)$, i.e., operational cohomology classes are compatible with the exterior product. This also show that the maps $c$ are linear over the coefficient ring $B^*$ of the theory.

Two of the three bivariant operations still make sense when restricted to the underlying cohomology theory. First of all, one defines the \emph{bivariant product} to simply be the composition of two bivariant classes. Moreover, for any morphism $f: X' \to X$ one can define the \emph{operational pullback}
\begin{equation*}
f^*: \op B^*(X) \to \op B^*(X')
\end{equation*}
simply by setting
\begin{equation*}
(f^*c)_{Y \to X'} = c_{Y \to X' \stackrel f \to X}.
\end{equation*}
One readily verifies that these operations produce operational cohomology classes.

\subsubsection*{Proof of Theorem \fref{UnivCoeff}}

We are now ready to prove the main theorem of this subsection. Let $X$ be a proper variety. We define the \emph{Kronecker duality map} 
\begin{equation*}
\op B^* (X) \to \Hom_{B^*}(B_*(X), B^*)
\end{equation*}
as the composition
\begin{equation*}
B_*(X) \stackrel c \to B_*(X) \stackrel {\pi_*} \to B_*(pt) = B^*,
\end{equation*}
where $\pi: X \to pt$ is the structure morphism. We begin with a simple observation.

\begin{lem}
Let $X$ be a proper variety satisfying the Künneth isomorphism criterion in \fref{UnivCoeff}. Then the Kronecker duality map is an injection.
\end{lem}
\begin{proof}
Let $Y \to X$ be a map, $\Gamma:  Y \to X \times Y$ be the graph embedding, and let $c$ be an operational cohomology class. Now we have the following diagram
\begin{center}
\begin{tikzpicture}[scale=2]
\node (X1) at (0,1) {$B_*(Y)$};
\node (X2) at (0,0) {$B_*(X) \otimes_{B^*} B_*(Y)$};
\node (Y1) at (2.5,1) {$B_*(Y)$};
\node (Y2) at (2.5,0) {$B_*(X) \otimes_{B^*} B_*(Y)$};
\node (Y3) at (4.5,0) {$B_*(Y)$};

\path[every node/.style={font=\sffamily\small}]
(X1) edge[->] node[right]{$\Gamma_*$} (X2) 
(X1) edge[->] node[above]{$c_{Y \to X}$} (Y1)
(Y1) edge[->] node[right]{$\Gamma_*$} (Y2)
(X2) edge[->] node[above]{$c_{X \times Y \to X}$} (Y2) 
(Y2) edge[->] node[above]{${\pi_2}_*$} (Y3)
(Y1) edge[->] node[above]{$1$} (Y3)
;
\end{tikzpicture}
\end{center}
where $\pi_2$ is the projection $X \times Y \to pt \times Y = Y$ inducing the proper pushforward
\begin{equation*}
{\pi_2}_* := \pi_* \otimes 1: B_*(X) \otimes_{B^*} B_*(Y) \to B_*(pt) \otimes_{B^*} B_*(Y) = B^*(Y),
\end{equation*}
by the basic compatibility properties of the exterior product with pushforwards. Moreover, by the operational cohomology axiom $C_4$, the map 
\begin{equation*}
c_{X \times Y \to X} : B_*(X) \otimes_{B^*} B_*(Y) \to B_*(X) \otimes_{B^*} B_*(Y) 
\end{equation*}
coincides with $c_{1} \otimes 1$, where $c_1$ is the homomorphism in the operational class $c$ corresponding to the identity map $X \to X$.

We have shown that $c_{Y \to X}$ coincides with the composition
\begin{equation*}
B_*(Y) \xrightarrow{\Gamma_*} B_*(X) \otimes_{B^*} B_*(Y) \xrightarrow{c_1 \otimes 1} B_*(X) \otimes_{B^*} B_*(Y) \xrightarrow{\pi_* \otimes 1} B_*(pt) \otimes B_*(Y). 
\end{equation*}
As $\Gamma_*$ does not depend on $c$, we have shown that the class $c$ is completely determined by its image under the Kronecker duality map, which is exactly what we wanted.
\end{proof}

\begin{rem}
In the above proof, we did not in fact use the Künneth isomorphism requirement in its full strength. Indeed, it would have been enough to require the Künneth morphism to be surjective. This is to make sure that the proper pushforward $\Gamma_*$ of an element $\beta \in B_*(Y)$ is of the form
\begin{equation*}
\Gamma_* (\beta) = \alpha_1 \times \beta_1 + \cdots + \alpha_r \times \beta_r,
\end{equation*}
and therefore its image in ${\pi_2}_* c$ is completely determined by the image of $c$ in the Kronecker morphism.
\end{rem}

In order to finish the proof of \fref{UnivCoeff}, it is enough to prove that any $B^*$-linear map $\psi : B_*(X) \to B^*$ gives rise to an operational cohomology class via the formula
\begin{equation*}
c_{Y \to X}  := B_* (Y) \stackrel {\Gamma_*} \to B_*(X) \otimes_{B^*} B_*(Y) \stackrel {\psi \otimes 1} \to B_*(Y). 
\end{equation*}
If the classes $c_{Y \to X}$ formed an operational cohomology class $c$, then the image of $c$ under the Kronecker duality map would be $\psi$ as the diagram
\begin{center}
\begin{tikzpicture}[scale=2]
\node (X0) at (-2, 1) {$B_*(X)$};
\node (X1) at (0,1) {$B_*(X) \otimes_{B^*} B_*(X)$};
\node (X2) at (0,0) {$B_*(X) \otimes_{B^*} B_*(pt)$};
\node (Y1) at (2.5,1) {$B_*(pt) \otimes_{B^*} B_*(X)$};
\node (Y2) at (2.5,0) {$B_*(pt) \otimes_{B^*} B_*(pt)$};

\path[every node/.style={font=\sffamily\small}]
(X0) edge[->] node[above]{$\Delta_*$} (X1)
(X0) edge[->] node[above]{$1$} (X2)
(X1) edge[->] node[right]{$1 \otimes \pi_*$} (X2) 
(X1) edge[->] node[above]{$\psi \otimes 1$} (Y1)
(Y1) edge[->] node[right]{$1 \otimes \pi_*$} (Y2)
(X2) edge[->] node[above]{$\psi \otimes 1$} (Y2) 
;
\end{tikzpicture}
\end{center}
commutes. This would prove the surjectivity of the Kronecker morphism. Note that this is where we require the Künneth morphism to be an isomorphism: if we cannot say that $B_*(X \times Y) =  B_*(X) \otimes_{B^*} B_*(Y)$, then we cannot define the function $\psi \otimes 1$.

\begin{proof}[Proof of Theorem \fref{UnivCoeff}.] We have to verify that the axioms $C_1$-$C_4$ are satisfied for a collection of morphism defined as above. 

($C_1$) Let $f: Y' \to Y$ be a proper morphism. Now we have the induced diagram
\begin{center}
\begin{tikzpicture}[scale=2]
\node (A1) at (0,1) {$B_*(Y')$};
\node (B1) at (2,1) {$B_*(X) \otimes_{B^*} B_*(Y')$};
\node (C1) at (4,1) {$B_*(Y')$};
\node (A2) at (0,0) {$B_*(Y)$};
\node (B2) at (2,0) {$B_*(X) \otimes_{B^*} B_*(Y)$};
\node (C2) at (4,0) {$B_*(Y)$};

\path[every node/.style={font=\sffamily\small}]
(A1) edge[->] node[right]{$f_*$} (A2)
(B1) edge[->] node[right]{$1 \otimes f_*$} (B2)
(A1) edge[->] node[above]{$\Gamma'_*$} (B1)
(A2) edge[->] node[above]{$\Gamma_*$} (B2)
(B1) edge[->] node[above]{$\psi \otimes 1$} (C1)
(B2) edge[->] node[above]{$\psi \otimes 1$} (C2)
(C1) edge[->] node[right]{$f_*$} (C2)
;
\end{tikzpicture}
\end{center}
The two small squares commute, and hence the big square commutes, proving ($C_1$). One proves the axioms ($C_2$) and ($C_3$) the same way, using the fact that smooth pullbacks and refined l.c.i. pullbacks commute with proper pushforwards.

($C_4$) Consider the composition $Y \times Z \to Y \to X$, where the first map is the canonical projection. Then the graph embedding $Y \times Z \to X \times Y \times Z$ equals $\Gamma \times 1_Z$, where $\Gamma$ is the graph embedding $Y \to X \times Y$. Thus the map associated to $Y \times Z \to X$ is given by
\begin{equation*}
B_*(Y \times Z) \stackrel{(\Gamma \times 1)_*} \to B_*(X) \otimes_{B^*} B_*(Y \times Z)  \stackrel{\psi \otimes 1} \to B_*(Y \times Z)
\end{equation*}
Now
\begin{equation*}
(\Gamma \times 1)_* (\alpha \times \beta) = \Gamma_*(\alpha) \times \beta,
\end{equation*} 
and 
\begin{equation*}
(\psi \otimes 1) (\Gamma_*(\alpha) \times \beta) = (\psi \otimes 1) \Gamma_*(\alpha) \times \beta,
\end{equation*}
so the collection of maps we have defined satisfies ($C_4$).
\end{proof}

This identification is functorial in the following sense. Suppose we have a morphism $f: X \to Y$ of proper varieties. Now there are two kinds of pullbacks one can think about in this situation. First of all, we have the usual operational pullbacks $f^*: \op B^* (Y) \to \op B^*(X)$. On the other hand, the morphism $f_*$ is proper, so we have the pullback
\begin{equation*}
(f_*)^*: \Hom_{B^*} (B_*(Y), B^*) \to \Hom_{B^*} (B_*(X), B^*) 
\end{equation*}
induced by the proper pushforward $f_*: B_*(X) \to B_*(Y)$. These work well with the Kronecker duality map, namely:

\begin{prop}
The two pullbacks above commute with the Kronecker morphism.
\end{prop}
\begin{proof}
Let $c \in \op B^*(Y)$ be an operational cohomology class. It is enough to show that the Kronecker image of $f^*c$ coincides with $\psi \circ f_*$, where $\psi$ is the Kronecker image of $c$. This follows directly from the commutativity of the diagram
\begin{center}
\begin{tikzpicture}[scale=2]
\node (A1) at (0,1) {$B_*(X)$};
\node (B1) at (2,1) {$B_*(Y) \otimes_{B^*} B_*(X)$};
\node (C1) at (4,1) {$B_*(X)$};
\node (D1) at (6,1) {$B^*$};
\node (A2) at (0,0) {$B_*(Y)$};
\node (B2) at (2,0) {$B_*(Y) \otimes_{B^*} B_*(Y)$};
\node (C2) at (4,0) {$B_*(Y)$};
\node (D2) at (6,0) {$B^*$};

\path[every node/.style={font=\sffamily\small}]
(A1) edge[->] node[right]{$f_*$} (A2)
(B1) edge[->] node[right]{$1 \otimes f_*$} (B2)
(A1) edge[->] node[above]{$\Gamma_*$} (B1)
(A2) edge[->] node[above]{$\Delta_*$} (B2)
(B1) edge[->] node[above]{$\psi \otimes 1$} (C1)
(B2) edge[->] node[above]{$\psi \otimes 1$} (C2)
(C1) edge[->] node[right]{$f_*$} (C2)
(C1) edge[->] node[above]{$\pi_{X*}$} (D1)
(C2) edge[->] node[above]{$\pi_{Y*}$} (D2)
(D1) edge[->] node[right]{$1$} (D2)
;
\end{tikzpicture}
\end{center}
where the top row is, by constructions of operational pullback and Kronecker map, just the Kronecker image of $f^*c$, and the bottom row is similarly the Kronecker image of $c$.
\end{proof}

\bibliographystyle{alphamod}
\bibliography{references}{}

\end{document}